\documentclass[twoside,11pt]{article}

\usepackage{blindtext}

\usepackage{jmlr2e}

\usepackage{amssymb,latexsym,amsmath}
\usepackage{epsfig}
\usepackage{caption}
 \usepackage{float}

 \usepackage{amsmath}
\usepackage{amsfonts}
\usepackage{array}
\usepackage{dsfont}
\usepackage{hyperref}
\usepackage{amssymb}

 \usepackage{algorithm}
\usepackage{algorithmic}
\usepackage{pifont}

\usepackage{makeidx}
\usepackage{graphicx}
\graphicspath{ {./images/} }

\allowdisplaybreaks

 \usepackage{graphicx,fancybox,latexsym,epsfig}
\usepackage{fancyhdr,amsmath,times,amsxtra,amssymb}
\usepackage{color}

\newtheorem{prop}{Proposition}[section]
\newtheorem{assumption}{Assumption}[section]

\def\argmin{\mathop{\rm arg\, min}}

\def\P{{\mathcal P}}

\allowdisplaybreaks

\usepackage{lastpage}
\jmlrheading{26}{2025}{1-\pageref{LastPage}}{8/24; Revised
3/25}{6/25}{24-1221}{Erhan Bayraktar and Ali D. Kara}
\ShortHeadings{Linear Approximations for Mean-Field Control }{Bayraktar and Kara}
\firstpageno{1}

\begin{document}

\sloppy
\title{Learning with Linear Function Approximations in Mean-Field Control
}
\author{ \name Erhan Bayraktar \email erhan@umich.edu \\
			\addr  Department of Mathematics,
     University of Michigan, Ann Arbor, MI, USA
     \AND 
     \name Ali Devran Kara \email akara@fsu.edu \\
     \addr Department of Mathematics, Florida State University, FL, USA}

\editor{Quanquan Gu}
\maketitle

\begin{abstract}
The paper focuses on mean-field type multi-agent control problems with finite state and action spaces where the dynamics and cost structures are symmetric and homogeneous, and are affected by the distribution of the agents. A standard solution method for these problems is to consider the infinite population limit as an approximation and use symmetric solutions of the limit problem to achieve near optimality. The control policies, and in particular the dynamics, depend on the population distribution in the finite population setting, or the marginal distribution of the state variable of a representative agent for the infinite population setting. Hence, learning and planning for these control problems generally require estimating the reaction of the system to all possible state distributions of the agents. To overcome this issue, we consider linear function approximation for the control problem and provide coordinated and independent learning methods. We rigorously establish error upper bounds for the performance of learned solutions. The performance gap stems from (i) the mismatch due to estimating the true model with a linear one, and (ii) using the infinite population solution in the finite population problem as an approximate control. The provided upper bounds quantify the impact of these error sources on the overall performance.
\end{abstract}


\section{Introduction}

The goal of the paper is to present various learning methods for mean-field control problems under linear function approximations and to provide provable error bounds for the learned solutions.

\subsection{Literature Review}
Learning for multi agent control problems is a practically relevant and a challenging problem where there has been as a growing interest in recent years. A general solution methodology for multi-agent control problems is difficult to obtain and the solution, in general, is intractable except for special information structures between the agents. We refer the reader to the survey paper by \cite{zhang2021multi} for a substantive summary of learning methods in the context of multi-agent decision making problems. 

In this paper, we study a particular case of multi-agent problems in which both the agents and their interactions are symmetric and homogeneous. For these mean-field type decision making problems, the agents are coupled only through the so-called mean-field term. These problems can be broadly divided into two categories; mean-field game problems where the agents are competitive and interested in optimizing their self objective functions, and mean-field control problems, where the agents are interested in a common objective function optimization. We cite some papers by  \cite{gomes2014mean,carmona2013probabilistic,bensoussan2013mean,tembine2013risk,huang2007large,anahtarci2022q,elie2020convergence,fu2019actor,guo2019learning, perrin2020fictitious,subramanian2019reinforcement,saldi2018markov, saldi2019approximate} and references therein, for papers in  mean-field game setting. We do not discuss these in detail as our focus will be on mean-field control problems which are significantly different in both analysis and the nature of the problems of interest.

 For mean-field control problems, where the agents are cooperative and work together to minimize (or maximize) a common cost (or reward) function, see \cite{bayraktar2018randomized,djete2022mckean, lauriere2014dynamic,  carmona2022convergence,pham2017dynamic,carmona2021convergence,germain2022numerical,bayraktar2021mean,bayraktar2021solvability} and references therein for the study of dynamic programming principle and learning methods in continuous time.  In particular, we point out the papers  \cite{lacker2017limit,djete2022mckean} which provide the justification for studying the centralized limit problem by rigorously connecting the large population decentralized setting and the infinite population limit problem.

For papers studying mean-field control in discrete time, we refer the reader to \cite{motte2022quantitative,bauerle2021mean,gu2021mean,gu2019dynamic,motte2022mean,carmona2019model}. \cite{motte2022quantitative,motte2022mean} study existence of solutions to the control problem in both infinite and finite population settings, and they rigorously establish the connection between the finite and infinite population problems.  \cite{bauerle2021mean} studies the finite population mean-field control problems and their infinite population limit, and provide solutions of the ergodic control problems for some special cases.

In the context of learning,  \cite{gu2021mean,gu2019dynamic} study dynamic programming principle and Q learning methods directly for the infinite population control problem. The value functions and the Q functions are defined for the lifted problem, where the state process is the measure-valued mean-field flow. They consider dynamics without common noise, and thus the learning problem from the perspective of a coordinator becomes a deterministic one. 

\cite{carmona2019model} also considers the limit (infinite population) problem and studies different classes of policies that achieve optimal performance for the infinite population (limit problem) and focuses on Q learning methods for the problem after establishing the optimality of randomized feedback policies for the agents. The learning problem considers the state as the measure valued mean-field term and defines a learning problem over the set of probability measures where various approximations are considered to deal with the high dimension issues.

\cite{angiuli2023convergence,angiuli2022unified} have studied learning methods for the mean-field game and control problems from  a joint lens. However, for the control setup, they consider a different control objective compared to the previously cited papers. In particular, they aim to optimize the asymptotic phase of the control problem where the agents are assumed to reach to their stationary distributions under joint symmetric policies. Furthermore, the agents only use their local state variables, and thus the objective is to find a stationary measure for the agents where the cost is minimized under this stationary regime. Since the agents only use their local state variables (and not the mean-field term) for their control, the authors can define a Q function over the finite state and action spaces of the agents.

\cite{pasztor2023efficient} consider a closely related problem to our setting, where they propose model-based learning methods for the mean-field control. Similar to \cite{gu2021mean,gu2019dynamic}, they directly work with the infinite population dynamics without analyzing the approximation consistency between the finite-population dynamics and their infinite-population counterpart. Furthermore, they restrict the dynamics to the models with additive noise, and the optimality search is within deterministic and Lipschitz continuous controls.

We also note that there are various studies that focus on the application of the mean-field modeling using numerical methods based on machine learning techniques, see e.g. the works by \cite{ruthotto2020machine,achdou2020mean,lauriere2022scalable}.

In this paper, we will consider the learning problem using an alternative formulation where the state is represented as the measure valued mean-field term. To approximate this uncountable space, and the cost and transition functions, different from the previous works in the mean-field control setting, we will consider linear function approximation methods. These methods have been studied well for single agent discrete time stochastic control problems. We cite papers by \cite{melo2008analysis,carvalho2020new,CsabaSmart,jin2020provably,meyn2024projected} in which reinforcement learning techniques are used to study Markov decision problems with continuous state spaces using linear function approximations.

{\bf Contributions.}

\begin{itemize}
\item In Section \ref{learning_sec}, we present the learning methods using linear function approximation. We focus on various scenarios.
\begin{itemize}
\item We first consider the ideal case where we assume that the team has infinitely many agents. For this case, we study; (i) learning by a coordinator who has access to information about  every agent in the team, and estimates a model from a data set by fitting a linear model that minimizes the $L_2$ distance between the training data and the estimate linear model, (ii) each agent estimates their own linear model using their local information via an iterative algorithm from a single sequence of data.
\item In Section \ref{ind_learn_fin}, we consider the practical case, where the team has finitely many agents, and they aim to estimate a linear model from a single sequence of data, using their local information variables.
\end{itemize} 
\item The methods we study in Section \ref{learning_sec} minimize the $L_2$ distance between the learned linear model and the actual model under a probability measure that depends on the training data.   However, to find upper bounds for the performance loss of the policies designed for the learned linear estimates in any scenario, we need uniform estimation errors rather than  $L_2$ estimation errors.   In Section \ref{unif_err_est}, we generalize  $L_2$ error bounds to uniform error bounds.

\item The proposed learning methods do not match the true model perfectly in general, due to linear approximation mismatch. Therefore, finally, in Section \ref{perf_loss}, we provide upper bounds on the performance of the policies that are designed for the learned models when they are applied on the true control problem. We note that the flow of the mean-field term is deterministic for infinitely many agents, and thus can be estimated using the dynamics without observing the mean-field term.   Therefore, for the execution of the policies we focus on two methods, (i) {\it open loop control}, where the agents only observe their local states and estimate the mean-field term with the learned dynamics, (ii) {\it closed loop control} where the agents observe both their local information and the mean-field term. For each of these execution procedures, we provide upper bounds for the performance loss. As in Section \ref{learning_sec}, we first consider the ideal case where it is assumed that the system has infinitely many agents. In this case, the error bound depends on the uniform model mismatch between the learned model and the true model. We then consider the case with finitely many agents. We assume that each agent follows the policy that they calculate considering the limit (infinite population) model. In this case, the error upper bounds depend on both the uniform model mismatch, and an empirical concentration bound since we estimate the finitely many agent model with the infinite population limit problem.
\end{itemize}

\subsection{Problem formulation.}

The dynamics for the model are presented as follows: suppose $N$ agents (decision-makers or controllers) act in a cooperative way to minimize a cost function, and the agents share a common state and an action space denoted by $\mathds{X}$ and $\mathds{U}$.  We assume that $\mathds{X}$ and $\mathds{U}$ are finite. We refer the reader to the paper by \cite{bayraktar2023finite}, for finite approximations of mean-field control problems where the state and actions spaces of the agents are continuous. For any time step $t$, and agent $i\in\{1,\dots,N\}$ we have
\begin{align}\label{dynamics}
x^i_{t+1}=f(x_t^i,u_t^i,\mu_{\bf x_t},w_t^i)
\end{align}
for a measurable function $f$, where $\{w_t^i\}$ denotes the i.i.d. idiosyncratic noise process.

 Furthermore, $\mu_{\bf x}\in\mathcal{P}_N(\mathds{X})$ denotes the empirical distribution of the agents on the state space $\mathds{X}$ such that for a given joint state of the team of agents ${\bf x}:=(x^1,\dots,x^N)\in \mathds{X}^N$
\begin{align*}
\mu_{\bf x}(\cdot):=\frac{1}{N}\sum_{i=1}^N\delta_{x^i}(\cdot)
\end{align*}
where $\delta_{x^i}$ represents the Dirac measure centered at $x^i$.
Throughout this paper, we use the notation
\begin{align*}
\mathds{X}^N:=\underbrace{\mathds{X}\times\dots\times\mathds{X}}_{\text{ $N$ times}}
\end{align*}
to denote the space of all joint state variables of the team
equipped with the product topology on $\mathds{X}^N$.
We further define $\mathcal{P}_N(\mathds{X})$, the set of all  empirical measures on $\mathds{X}$ constructed using sequences of $N$ states in $\mathds{X}$, such that
\begin{align*}
\mathcal{P}_N(\mathds{X}):=\{\mu_{\bf x}: {\bf x}=(x^1,\dots,x^N)\in\mathds{X}^N\}.
\end{align*}
Note that $\mathcal{P}_N(\mathds{X})\subset \P(\mathds{X})$ where $\P(\mathds{X})$ denotes the set of all probability measures on $\mathds{X}$ equipped with the weak convergence topology.

Equivalently, the next state of the agent $i$ is determined by some stochastic kernel, that is, a regular conditional probability distribution:
\begin{align}\label{kernel_common_noise}
\mathcal{T}(\cdot|x_t^i,u_t^i,\mu_{\bf x_t}).
\end{align}
At each time stage $t$, each agent receives a cost determined by a measurable stage-wise cost function $c:\mathds{X}\times\mathds{U}\times \mathcal{P}_N(\mathds{X})\to \mathds{R}$. If the state, action, and empirical distribution of the agents are given by $x_t^i,u_t^i,\mu_{\bf x_t}$, then the agent receives the cost. 
\begin{align*}
c(x_t^i,u_t^i,\mu_{\bf x_t}).
\end{align*}
For the remainder of the paper, by an abuse of notation, we will sometimes denote the dynamics in terms of the vector state and action variables, ${\bf x}=(x^1,\dots,x^N)$, and ${\bf u}=(u^1,\dots,u^N)$, and vector noise variables ${\bf w}=(w^1,\dots,w^N)$ such that
\begin{align*}
{\bf x_{t+1}}=f({\bf x_t,u_t,w_t}).
\end{align*}
For the initial formulation, every agent is assumed to know the state and action variables of every other agent. We define an admissible policy for an agent $i$, as a sequence of functions $\gamma^i:=\{\gamma^i_t\}_t$, where $\gamma^i_t$ is a $\mathds{U}$-valued (possibly randomized) function which is measurable with respect to the $\sigma$-algebra generated by 
\begin{align}\label{info}
I_t=\{{\bf x}_0,\dots,{\bf x}_t,{\bf u}_0,\dots,{\bf u}_{t-1}\}.
\end{align}
Accordingly, an admissible {\it team} policy, is defined as $\gamma:=\{\gamma^1,\dots,\gamma^N\}$, where $\gamma^i$ is an admissible policy for the agent $i$. In other words, agents share the complete information.

The objective of the agents is to minimize the following cost function
\begin{align*}
J^N_\beta({\bf x}_0,\gamma)=\sum_{t=0}^{\infty}\beta^tE_\gamma\left[{\bf c}({\bf x}_t,{\bf u}_t)\right]
\end{align*}
where $E_\gamma$ denotes the expectation with respect to the probability measure induces by the team policy $\gamma$, and where
\begin{align*}
{\bf c}({\bf x}_t,{\bf u}_t):=\frac{1}{N}\sum_{i=1}^Nc(x^i_t,u^i_t,\mu_{{\bf x}_t}).
\end{align*}
The optimal cost is defined by
\begin{align}\label{opt_cost}
J_\beta^{N,*}({\bf x}_0):=\inf_{\gamma\in\Gamma}J_\beta^N({\bf x}_0,\gamma)
\end{align}
where $\Gamma$ denotes the set of all admissible team policies.


We note that this information structure (\ref{info}) will be our benchmark for evaluating the performance of the approximate solutions using simpler information structures presented in the paper. In other words, the value function that is achieved when the agents share full information and full history will be taken to be our reference point for simpler information structures. 

For example, one immediate observation is that the problem under full information sharing can be reformulated as a centralized control problem where the state and action spaces are  $\mathds{X}^N$ and $\mathds{U}^N$. Therefore, one can consider Markov policies such that $I_t=\{{\bf x}_t\}$ without loss of optimality. 

However, if the problem is modeled as an MDP with state space $\mathds{X}^N$ and action space $\mathds{U}^N$, we face some computational challenges:
\begin{itemize}
\item[ (i)] the curse of dimensionality when $N$ is large, since $\mathds{X}^N$ and $\mathds{U}^N$ might be too large even when $\mathds{X,U}$ are of manageable size,
\item[(ii)] the curse of coordination: even if the optimal team policy is found, its execution at the agent level requires coordination among the agents. In particular, the agents may need to follow asymmetric policies to achieve optimality, even though we assume full symmetry for the dynamics and the cost models. 
The following simple example from \cite{bayraktar2023infinite} shows that the agents may need to follow asymmetric policies to achieve optimality which requires coordination among the agents.
\end{itemize}

\begin{example}\label{counter_ex}
Consider a team control problem with two agents, i.e. $N=2$. We assume that $\mathds{X}=\mathds{U}=\{0,1\}$. The stage wise cost function of the agents is defined as
\begin{align*}
c(x,u,\mu_{\bf x}) = \|\mu_{\bf x}-\bar{\mu})\|
\end{align*} 
where 
\begin{align*}
\bar{\mu}=\frac{1}{2}\delta_0 + \frac{1}{2}\delta_1.
\end{align*}
In words, the state distribution should be distributed equally over the state space $\{0,1\}$ for minimal stage-wise cost. For the dynamics we assume a deterministic model such that
\begin{align*}
x_{t+1}=u_t.
\end{align*}
In words, the action of an agent purely determines the next state of the same agent. The goal of the agents is to minimize
\begin{align*}
\sum_{t=0}^{\infty}\beta^t E^{g^1,g^2}\left[\frac{c(x_t^1,u_t^1,\mu_{\bf x_t}) +c(x_t^2,u_t^2,\mu_{\bf x_t})  }{2}\right]
\end{align*}
for some initial state values ${\bf{x}}_0=[x_0^1,x_0^2]$, by choosing policies $g^1,g^2$. The expectation is over the possible randomization of the policies.  We assume full information sharing such that every agent has access to the state and action information of the other agent. 

We let the initial states be $x_0^1=x_0^2=0$. An optimal policy for the agents for the problem is given by
\begin{align*}
&g^1(0,0)= 0, \qquad g^2(0,0)=1\\
&g^1(0,1)= 0 ,\qquad g^2(0,1)=1\\
 &g^1(1,0)= 1 ,\qquad g^2(1,0)=0\\
&g^1(1,1)= 1 ,\qquad g^2(1,1)=0
\end{align*}
which always spreads the agents equally over the state space. One can realize that, when the agents are positioned at either $(0,0)$ or $(1,1)$, they have to use personalized policies to decide on which one to be placed at $0$ or $1$. 

For any symmetric policy $g^1(x^1,x^2)=g^2(x^1,x^2)=g(x^1,x^2)$, including the randomized ones, there will always be cases with strict positive probability, where the agents are positioned at the same state, and thus the performance will be strictly worse than the optimal performance.
\end{example}

A standard approach to deal with mean-field control problems when $N$ is large is to consider the infinite population problem, i.e. taking the limit $N\to \infty$. A propagation of chaos argument can be used to show that in the limit, the agents become asymptotically independent. Hence, the problem can be formulated from the perspective of a representative single agent.  This approach is suitable to deal with coordination challenges, as the correlation between the agents vanish in the limit, and thus the symmetric policies can achieve optimal performance for the infinite population control problem. In particular, for Example \ref{counter_ex} in the infinite population setting, the optimal policy is to follow a randomized policy such that $Pr(u=1)=Pr(u=0)=\frac{1}{2}$. We will introduce  the limit problem  in Section \ref{limit_problem} and make the connections between the limit problem and the finite population problem rigorous.

\subsection{Preliminaries.} 

Recall that we assume that the state and action spaces of agents $\mathds{X,U}$ are finite (see \cite{bayraktar2023finite} for finite approximations of continuous space mean-field control problems). 

{\bf Note.} Even though we assume that $\mathds{X}$ and $\mathds{U}$ are finite, we will continue using integral signs instead of summation signs for expectation computations due to notation consistency, by simply considering Dirac delta measures.

We metrize $\mathds{X}$ and $\mathds{U}$ so that $d(x,x')=1$ if $x\neq x'$ and $d(x,x')=0$ otherwise. Note that with this metric, for any $\mu,\nu\in\P(\mathds{X})$ and for any coupling $Q$ of $\mu,\nu$, we have that
\begin{align*}
E_Q\left[|X-Y|\right]= P_Q(X\neq Y)
\end{align*}
which in particular implies  via the optimal coupling that 
\begin{align*}
W_1(\mu,\nu)=\|\mu-\nu\|_{TV}
\end{align*}
where $W_1$ denotes the first order Wasserstein distance, or the Kantorovich–Rubinstein metric, and $\|\cdot\|_{TV}$ denotes the total variation norm for signed measures.

Note further that for measures defined on finite spaces, we have that
\begin{align}\label{norm}
\|\mu-\nu\|_{TV}=\frac{1}{2}\|\mu-\nu\|_1=\frac{1}{2}\sum_x\left|\mu(x)-\nu(x)\right|.
\end{align}
Hence, in what follows we will simply write $\|\mu-\nu\|$ to refer to the distance between $\mu$ and $\nu$, which may correspond to the total variation distance,  the first order Wasserstein metric, or the normalized $L_1$ distance.

We also define the following Dobrushin coefficient for the kernel $\mathcal{T}$:
\begin{align}\label{dob}
\sup_{\mu,\gamma,x,\hat{x}}\|\int_{\mathds{U}}\mathcal{T}(\cdot|x,u,\mu)\gamma(du|x)-\int_{\mathds{U}}\mathcal{T}(\cdot|\hat{x},u,\mu)\gamma(du|\hat{x})\|=:\delta_T
\end{align}
Realize that we always have $\delta_T\leq 1$. In certain cases, we can also have strict inequality, e.g. if there exists some $x^*\in\mathds{X}$ such that
\begin{align*}
\mathcal{T}(x^*|x,u,\mu)<1-\alpha,\quad \forall x,u,\mu
\end{align*}
 then one can show that $\delta_T\leq1-\alpha<1.$

\subsection{Measure Valued Formulation of the Finite Population Control Problem}\label{fin_meas_sec}
For the remaining part of the paper, we will often consider an alternative formulation of the control problem for the finitely many agent case where the controlled process is the state distribution of the agents, rather than the state vector of the agents. We refer the reader to \cite{bayraktar2023finite} for the full construction; in this section, we will give an overview.

We define an MDP for the distribution of the agents, where the control actions are the joint distribution of the state and action vectors of the agents. 

We let the state space be  $\mathcal{Z}=\P_N(\mathds{X})$ which is the set of all empirical measures on $\mathds{X}$ that can be constructed using the state vectors of  $N$-agents. In other words, for a given state vector ${\bf x}=\{x^1,\dots,x^N\}$, we consider $\mu_{\bf x}\in \P_N(\mathds{X})$ to be the new state variable of the team control problem.

The admissible set of actions for some state $\mu\in\mathcal{Z}$, is denoted by $U(\mu)$, where
\begin{align}\label{add_act}
U(\mu)=\{\Theta\in \P_N(\mathds{U}\times\mathds{X})|\Theta(\mathds{U},\cdot)=\mu(\cdot)\},
\end{align}
that is, the set of actions for a state $\mu$, is the set of all joint empirical measures on $\mathds{X}\times\mathds{U}$ whose marginal on $\mathds{X}$ coincides with $\mu$.

We equip the state space $\mathcal{Z}$, and the action sets $U(\mu)$, with the norm $\| \cdot\|_{TV}$ (see (\ref{norm}))  .

One can show that \cite{bayraktar2023finite, bayraktar2023infinite} the empirical distributions of the states of agents $\mu_t$, and of the joint state and actions $\Theta_t$ define a controlled Markov chain such that
\begin{align}\label{trans_measure}
&Pr(\mu_{t+1}\in B|\mu_t,\dots\mu_0,\Theta_t,\dots,\Theta_0)=Pr(\mu_{t+1}\in B|\mu_t,\Theta_t)\nonumber\\
&:=\eta(B|\mu_t,\Theta_t)
\end{align}
where $\eta(\cdot|\mu,\Theta)\in \P(\P_N(\mathds{X}))$ is the transition kernel of the centralized measure valued MDP, which is induced by the dynamics of the team problem.

We define the stage-wise cost function $k(\mu,\Theta)$ by
\begin{align}\label{cost_measure}
k(\mu,\Theta):=\int c(x,u,\mu)\Theta(du,dx)=\frac{1}{N}\sum_{i=1}^Nc(x^i,u^i,\mu).
\end{align}



Thus, we have an MDP with state space $\mathcal{Z}$, action space $\cup_{\mu\in\mathcal{Z}}U(\mu)$, transition kernel $\eta$ and the stage-wise cost function $k$.

We define the set of admissible policies for this measured valued MDP as a sequence of functions $g=\{g_0,g_1,g_2,\dots\}$ such that at every time $t$, $g_t$ is measurable with respect to the $\sigma$-algebra generated by the information variables
\begin{align*}
\hat{I}_t=\{\mu_0,\dots,\mu_t,\Theta_0,\dots,\Theta_{t-1}\}.
\end{align*}
We denote the set of all admissible control policies by $G$ for the measure valued MDP.

In particular, we define the infinite horizon discounted expected cost function under a policy $g$ by
\begin{align*}
K^N_\beta(\mu_0,g)=E_{\mu_0}^\eta\left[ \sum_{t=0}^{\infty}\beta^tk(\mu_t,\Theta_t)\right].
\end{align*}
We also define the optimal cost by
\begin{align}\label{opt_cost2}
K_\beta^{N,*}(\mu_0)=\inf_{g\in G}K^N_\beta(\mu_0,g).
\end{align}

The following result shows that this formulation is without loss of optimality:
\begin{theorem}[\cite{bayraktar2023finite}]
Under Assumption \ref{main_assmp}, for any ${\bf x}_0$ that satisfies $\mu_{{\bf x}_0}=\mu_0$, that is for any ${\bf x}_0$ with distribution $\mu_0$, we have that
\begin{itemize}
\item[i.)]
\begin{align*}
K_\beta^{N,*}(\mu_{0})=J_\beta^{N,*}({\bf x}_0).
\end{align*}

\item[ii.)] There exists a stationary and Markov optimal policy $g^*$ for the measure valued MDP, and  using $g^*$, every agent can construct a policy $\gamma^i:\mathds{X}\times \P_N(\mathds{X})\to\mathds{U}$ such that for $\gamma:=\{\gamma^1,\gamma^2,\dots,\gamma^N\}$, we have that
\begin{align*}
J^N_\beta({\bf x_0},\gamma)=J_\beta^{N,*}({\bf x_0}).
\end{align*}
That is, the policy obtained from the measure valued formulation attains the optimal performance for the original team control problem.
\end{itemize}
\end{theorem}

\subsection{Mean-field Limit Problem}\label{limit_problem}
We now introduce the control problem for infinite population teams, i.e. for $N\to\infty$. For some agent $i\in \mathds{N}$, we define the dynamics as
\begin{align*}
x^i_{t+1}=f(x_t^i,u_t^i, \mu_t^i,w_t^i)
\end{align*}
where $x_0\sim\mu_0$ and $\mu_t^i=\mathcal{L}(X_t^i)$ is the law of the state at time $t$. The agent tries to minimize the following cost function:
\begin{align*}
J_\beta^{\infty}(\mu_0,\gamma)=\sum_{t=0}^\infty \beta^t E\left[c(X_t^i,U_t^i,\mu_t^i) \right]
\end{align*}
where $\gamma=\{\gamma_t\}_t$ is an admissible policy such that $\gamma_t$ is measurable with respect to the information variables 
\begin{align*}
I_t^i=\left\{ x_0^i,\dots,x_t^i,u_0^i,\dots,u_{t-1}^i,\mu_0^i,\dots,\mu_t^i  \right\}.
\end{align*}
Note that the agents are no longer correlated and they are indistinguishable. Hence, in what follows we will drop the dependence on $i$ when we refer to the infinite population problem. 

The problem is now a single agent control problem; however, the state variable is not Markovian. However, we can reformulate the problem as an MDP by viewing the state variable as the measure valued $\mu_t$. 

We let the state space to be $\P(\mathds{X})$. Different from the measure valued construction we have introduced in Section \ref{fin_meas_sec}, we let the action space to be $\Gamma = \P(\mathds{U})^{|\mathds{X}|}$. In particular, an action $\gamma(\cdot|x)\in\Gamma$ for the team is a randomized policy at the agent level. We equip $\Gamma$ with the product topology, where we use the weak convergence for each coordinate.  We note that each action $\gamma(du|x)$ and state $\mu(dx)$ induce a distribution on $\mathds{X}\times\mathds{U}$, which we denote by $\Theta(du,dx)=\gamma(du|x)\mu(dx)$.
 
 Recall the notation in (\ref{kernel_common_noise}); at time $t$, we can use the following stochastic kernel for the dynamics:
\begin{align*}
x_{t+1}\sim \mathcal{T}(\cdot|x_t,u_t,\mu_t)
\end{align*}
which is induced by the idiosyncratic noise $w^i_t$. Hence, we can define
\begin{align}\label{inf_kernel}
\mu_{t+1}=F(\mu_t,\gamma_t):= \int \mathcal{T}(\cdot|x,u,\mu_t)\gamma_t(du|x)\mu_t(dx).
\end{align}
Note that the dynamics are deterministic for the infinite population measure valued problem. Furthermore, we can define the stage-wise cost function as
\begin{align}\label{inf_cost}
k(\mu,\gamma):=\int c(x,u,\mu)\gamma(du|x)\mu(dx).
\end{align}
Hence, the problem is a deterministic MDP for the measure valued state process $\mu_t$. A policy, say $g:\P(\mathds{X})\to \Gamma$ for the measure-valued MDP, can be written as
\begin{align*}
g(\mu)=\gamma(du|x)
\end{align*}
for some $\mu\in\P(\mathds{X})$. That is, an agent observes $\mu$ and chooses their actions as an agent-level randomized policy $\gamma(du|x)$. 

We reintroduce the infinite horizon discounted cost of the agents for the measure valued formulation:
\begin{align*}
K_\beta(\mu_0,g) = \sum_{t=0}^\infty \beta^t k(\mu_t,\gamma_t)
\end{align*}
for some initial mean-field term $\mu_0$ and under some policy $g$. Furthermore, the optimal policy is denoted by 
\begin{align*}
K_\beta^*(\mu_0) = \inf_g K_\beta(\mu_0,g).
\end{align*}
At a given time $t$, the pair $(x_t,\mu_t)$ can be used as sufficient information for decision making by the agent $i$. Furthermore, if the model is fully known by the agents, then the mean-field flow $\mu_t$ can be perfectly estimated if every agent agrees and follows the same policy $g(\mu)$, since the dynamics of $\mu_t$ is deterministic.

We note that  for the infinite population control problem, the coordination requirement between the agents may be relaxed, though cannot be fully abandoned in general (see Section \ref{lim_dec}).  In particular, if the agents agree on a common policy $g(\mu)=\gamma(du|x,\mu)$, then for the execution of this policy, no coordination or communication is needed since every agent can estimate the mean-field term $\mu_t$ independently and perfectly.  Furthermore, every agent can use the same agent-level policy $\gamma(du|x,\mu)$ symmetrically, without any coordination with the other agents.

The following result makes the connection between the finite population and the infinite population control problem rigorous \cite{motte2022mean,bauerle2021mean,bayraktar2023infinite}.

\begin{assumption}\label{main_assmp}
\begin{itemize}
\item[i.] For the transition kernel $\mathcal{T}(\cdot|x,u,\mu)$ (see (\ref{kernel_common_noise}))
\begin{align*}
\|\mathcal{T}(\cdot|x,u,\mu) - \mathcal{T}(\cdot|x,u,\mu')\| \leq K_f \|\mu-\mu'\|
\end{align*} 
for some $K_f<\infty$, for each $x,u$ and for every $\mu,\mu'\in\P(\mathds{X})$.
\item[ii] $c$ is  Lipschitz in $\mu$ such that
\begin{align*}
|c(x,u,\mu)-c(x,u,\mu')|\leq K_c \|\mu-\mu'\|
\end{align*}
for some $K_c<\infty$.
\end{itemize}
\end{assumption}

\begin{theorem}
Under Assumption \ref{main_assmp}, the following holds:
\begin{itemize}
\item[i.] For any $\mu_0^N \to \mu_0$, 
\begin{align*}
\lim_{N\to\infty} K_\beta^{N,*}(\mu_0^N) = K_\beta^{\infty,*}(\mu_0).
\end{align*}
That is, the optimal value function of the finite population control problem converges to that of the infinite population control problem as $N\to\infty$.
\item[ii.] Suppose each agent solves the infinite population control problem given in (\ref{inf_kernel}) and (\ref{inf_cost}), and constructs their policies, say $$g_\infty(\mu)=\gamma_\infty(du|x,\mu).$$ If they follow the infinite population solution in the finite population control problem, for any $\mu_0^N \to \mu_0$ we then have
\begin{align*}
\lim_{N\to\infty} K_\beta^{N}(\mu_0^N,g_\infty) = K_\beta^{\infty,*}(\mu_0).
\end{align*}
That is, the symmetric policy constructed using the infinite population problem is near optimal for finite but sufficiently large populations.  
\end{itemize}
\end{theorem}

\begin{remark}
The result has significant implications for the computational challenges we have mentioned earlier. Firstly, the second part of the result states that if the number of agents is large enough, then the symmetric policy obtained from the limit problem is near optimal. Hence,  the agents can use symmetric policies without coordination, solving their control problems as long as they have access to the mean-field term and their local state. Secondly, note that the flow of the mean-field term $\mu_t$ (\ref{inf_kernel}) is deterministic if there is no common noise affecting the dynamics. Thus, agents can estimate the marginal distribution of their local state variables $x_t^i$, without observing the mean-field term if they know the dynamics. In particular, without the common noise, the local state of the agents  and  the initial mean-field term $\mu_0$ are sufficient information for near optimality.

However, as we will see in what follows, to achieve near optimal performance, agents must agree on a particular policy $g(\mu)=\gamma(du|x,\mu)\mu(dx)$. In particular, if the optimal infinite population policy is not unique, and the agents apply different optimal policies without coordination, the results of the previous results might fail. Hence, coordination cannot be fully ignored.
\end{remark}

\subsection{Limitations of Full Decentralization}\label{lim_dec} We have argued in the previous section that the team control problem can be solved near optimally by using the infinite population control  solution. Furthermore, if the agents agree on the application of a common optimal policy, the resulting team policy can be executed independently in a decentralized way and achieves near-optimal performance. 

The following example shows that if the agents do not coordinate on which policy to follow, i.e. if they are fully decentralized, then the resulting team policy will not achieve the desired outcome.
\begin{example}\label{dec_counter}
Consider a team control problem with infinite population where $\mathds{X}=\mathds{U}=\{0,1\}$. The stage wise cost function of the agents is defined as
\begin{align*}
c(x,u,\mu) = \begin{cases}
\|\mu-\bar{\mu}_1\|& \text{ if } \mu(0)\leq \frac{3}{4}\\
\|\mu-\bar{\mu}_2\|& \text{ otherwise}
\end{cases}
\end{align*} 
where 
\begin{align*}
&\bar{\mu}_1=\frac{1}{2}\delta_0 + \frac{1}{2}\delta_1\\
&\bar{\mu}_2 =\delta_0.
\end{align*}
In words, the state distribution should be either be distributed equally over the state space $\{0,1\}$  or it should fully concentrate in state $0$ for minimal stage-wise cost. One can check that the cost function satisfies Assumption  \ref{main_assmp} for some $K_c<\infty$ (e.g. $K_c=1$).  For the dynamics we assume a deterministic model such that
\begin{align*}
x_{t+1}=u_t.
\end{align*}
In words, the action of an agent purely determines the next state of the same agent. The goal of the agents is to minimize
\begin{align*}
K_\beta(\mu_0 ,g)  =\limsup_{N\to \infty} \sum_{t=0}^\infty\beta^t E [\frac{1}{N}\sum_{i=1}^N c(x_t^i,u_t^i,\mu_{\bf x_t})] 
\end{align*}
where the initial distribution is given by $\mu_0 = \frac{1}{2}\delta_0 + \frac{1}{2}\delta_1$.

It is easy to see that there are two possible optimal policies for the agents $g_1(\mu)= \gamma_1(du|x,\mu)\mu(dx)$ and $g_2(\mu)= \gamma_2(du|x,\mu)\mu(dx)$  where
\begin{align*}
&\gamma_1(\cdot|x)= \frac{1}{2}\delta_0(\cdot) + \frac{1}{2} \delta_1(\cdot)\\
&\gamma_2(\cdot|x) = \delta_0(\cdot).
\end{align*} 
If all the agents coordinate and apply either $g_1$ or $g_2$ all together, the realized costs will be $0$, i.e.
\begin{align*}
K_\beta(\mu_0,g_1)=K_\beta(\mu_0,g_2) =0.
\end{align*}

However, if the agents do not coordinate and pick their policies from $g_1,g_2$ randomly, the cost incurred will be strictly greater than $0$. For example, assume that any given agent decides to use $g_1$ with probability $0.5$ and the policy $g_2$ with probability $0.5$. Then the resulting policy, say $\hat{g}$ will be such that 
\begin{align*}
\hat{g}(\mu)= \left(\frac{1}{2}\gamma_1(du|x) + \frac{1}{2}\gamma_2(du|x)\right)\mu(dx)
\end{align*}

Thus, at every time step $t \geq 1$, $\frac{1}{4}$ of the agents will be in state $1$ and $\frac{3}{4}$ of the agents will be in state $0$, hence the total accumulated cost of the resulting policy $\hat{g}$ will be 
\begin{align*}
K_\beta(\mu_0,\hat{g}) = \sum_{t=1}^\infty \beta^t \frac{1}{4} = \frac{\beta}{1-4\beta}>0.
\end{align*}
Thus, we see that if the optimal policy for the mean-field control problem is not unique, the agents cannot follow fully decentralized policies, and they need to coordinate at some level. For this problem, if they agree initially on which policy to follow, then no other communication is needed afterwards for the execution of the decided policy. Nonetheless, an initial agreement and coordination is needed to achieve the optimal performance.
\end{example}

We note that the issue with the previous example results from the fact that the optimal policy is not unique. If the optimal policy can be guaranteed to be unique, then the agents can act fully independently.






\section{Learning for Mean-field Control with Linear Approximations} \label{learning_sec}

We have seen in the previous sections that in general there are limitations for full decentralization, and that a certain level of coordination is required for optimal or near optimal performance during control. In this section, we will study the learning problem in which neither the agents nor the coordinator know the dynamics and aims to learn the model or optimal decision strategies from the data.

We have observed that the limit problem introduced in Section \ref{limit_problem} can be seen as a deterministic centralized control problem. In particular, if the model is known, and once it is coordinated which control strategy to follow, the agents do not  need further communication or coordination  to execute the optimal control. Each agent can simply apply an open-loop policy using only their local state information, and the mean-field term can be estimated perfectly, if every agent is following the same policy. However, to estimate the deterministic mean-field flow $\mu_t$, the model must be known. For problems where the model is not fully known, the open-loop policies will not be applicable. 

Our goal in this section is to present various learning algorithms to learn the dynamics and cost model of the control problem. We will first focus on the idealized scenario, where we assume that there exist infinitely many agents on the team. For this case, we provide two methods; (i) the first one where a coordinator has access to all information of every agent, and decides on the exploration policy, and (ii) the second one where each agent learns the model on their own by tracking their local state and the mean-field term. However, the agents need to coordinate for the exploration policy through a common randomness variable to induce stochastic dynamics for better exploration. Next, we study the realistic setting where the team has large but finitely many agents. For this case, we only consider an independent learning method where the agents learn the model on their own using their local information variables. 



Before we present our learning algorithms, we note that the space $\P(\mathds{X})$ is uncountable even under the assumption that $\mathds{X}$ is finite. Therefore, we will focus on finite representations of the cost function $c(x,u,\mu)$ and the kernel $\mathcal{T}(\cdot|x,u,\mu)$. In particular, we will try to learn the functions of the following form
\begin{align}\label{lin_form}
&c(x,u,\mu) =  {\bf \Phi}_{(x,u)}^\intercal(\mu) {\bf \theta}_{(x,u)}\nonumber\\
&\mathcal{T}(\cdot|x,u,\mu) = {\bf \Phi}_{(x,u)}^\intercal(\mu) {\bf Q}_{(x,u)}(\cdot)
\end{align}
where ${\bf \Phi}_{(x,u)}(\mu) =[ \Phi^1_{(x,u)}(\mu),\dots,\Phi^d_{(x,u)}(\mu)]^\intercal$,  for a set of linearly independent functions $ \Phi^j_{(x,u)}(\mu):\P(\mathds{X})\to \mathds{R}$ for each pair $(x,u)$, for some $d<\infty$. We assume that the basis functions ${\bf \Phi}_{(x,u)}(\mu)$ are known and the goal is to learn the parameters $\theta_{(x,u)}$ and ${\bf Q}_{(x,u)}(\cdot)$.  We assume $\theta_{(x,u)}\in\mathds{R}^d$, and ${\bf Q}_{(x,u)}(\cdot) = [Q^1_{(x,u)}(\cdot), \dots, Q_{(x,u)}^d(\cdot)]$ is a vector of unknown signed measures on $\mathds{X}$. 

In what follows, we will assume that the basis functions, $\Phi^j_{(x,u)}(\cdot)$ are uniformly bounded. Note that this is without loss of generality. 
\begin{assumption}\label{bdd_assmp}
We assume that 
\begin{align*}
\|\Phi^j_{(x,u)}(\cdot)\|_\infty\leq 1
\end{align*}
for every $(x,u)$ pair, and for all $j\in\{1,\dots,d\}$.
\end{assumption}

For the rest of the paper, we will use $\theta_{(x,u)}$ and $\theta(x,u)$ interchangeably; similarly we will use ${\bf Q}_{(x,u)}$ and ${\bf Q}(x,u)$ interchangeably.

\begin{remark}
We note that we {\bf do not} assume that the model and the cost function have the linear form given in (\ref{lin_form}). However, we will aim to learn and estimate models among the class of linear functions presented in (\ref{lin_form}). We will later analyze error bounds for the case where the actual model is not linear and thus the learned model does not perfectly match the true model and study the performance loss when we apply policies that are learned for the linear model.
\end{remark}

\subsection{Coordinated Learning with Linear Function Approximation for Infinitely Many Players}\label{coor_learn_sec}
In this section, we will consider an idealized scenario, where there are infinitely many agents, and a coordinator learns the model by linear function approximation.

\noindent{\bf Data collection.} For this section, we assume that there exists a training set $T$ that consists of a time sequence of length $M$. The training set is assumed to be coming from an arbitrary sequence of data. The data at each time stage contains  
\begin{align*}
x^i,u^i,X_1^i,c(x^i,u^i,\mu),\mu
\end{align*}
for all the agents present in the team, $i\in\{1,\dots,N,\dots\}$, where the agents' states are distributed according to $\mu$ at the given time step.
 That is, every data point includes the current state and action, the one-step ahead state, the stage-wise cost realization, and the mean-field term for every agent. Furthermore, we assume the ideal scenario where there are infinitely many agents. Hence, at every time step, the coordinator has access to infinitely many data points where the spaces $\mathds{X},\mathds{U}$ are finite. The coordinator then has access to infinitely many sample transitions observed under $(x, u, \mu)$, and thus, the kernel $\mathcal{T}(\cdot|x,u,\mu)$ can be perfectly estimated for every $x$ such that $\mu(x)>0$ and $\gamma(u|x)>0$, via empirical measures. Here, $\gamma$ represents the exploration policy of the agents.   We assume the following:
 \begin{assumption}\label{pos_expo}
     For any $x\in\mathds{X}$, the exploration policy for every agent puts positive probability on to every control action such that
     \begin{align*}
         \gamma(u|x)>0, \text{ for all } (x,u)\in\mathds{X\times U}.
     \end{align*}
 \end{assumption}

 We define the following sets for which the model and the cost functions can be learned perfectly within the training data: let $x\in\mathds{X}$, we define
 \begin{align}\label{learned_set}
P_x:=\{\mu\in T: \mu(x)>0\}.
 \end{align}
$P_x\subset \P(\mathds{X})$ denotes the set of probability measures which assign positive measure to a particular state $x\in\mathds{X}$ that are also in the training data for the mean-field terms. In particular, for a given $(x,u)$ pair,  the kernel $\mathcal{T}(\cdot|x,u,\mu)$ and the cost $c(x,u,\mu)$ can be  learned perfectly for every $\mu\in P_x$ with Assumption \ref{pos_expo}.

 For a given $x\in\mathds{X}$, we denote by $M_x$ the number of mean-field terms within the set $P_x$ (see (\ref{learned_set})). 
 For every $(x,u)\in\mathds{X}\times\mathds{U}$ pair, the coordinator aims to find ${\bf \theta}_{(x,u)}$ and ${\bf Q}_{(x,u)}$ such that 
\begin{align*}
&\frac{1}{M_x}\sum_{j=1}^{M_x}\left| c(x,u,\mu_j) - {\bf \Phi^\intercal_{(x,u)}}(\mu_j) {\bf \theta_{(x,u)}}\right|^2\\
&\frac{1}{M_x}\sum_{j=1}^{M_x}\left\| \mathcal{T}(\cdot |x,u,\mu_j) - {\bf \Phi^\intercal_{(x,u)}}(\mu_j) {\bf Q_{(x,u)}(\cdot)}\right\|
\end{align*}
is minimized.


The least squares linear models can be estimated in closed form for ${\bf \theta}_{(x,u)}$ and ${\bf Q}_{(x,u)}(\cdot)$ using the training data. We define the following vector and matrices to present the closed form solution in  a more compact form: for each $(x,u)\in\mathds{X}\times\mathds{U}$ we introduce ${\bf b}_{(x,u)} \in \mathds{R}^{M_x}$, and ${\bf d}_{(x,u)} \in \mathds{R}^{M_x\times |\mathds{X}|}$,
\begin{align}\label{lin_fit1}
{\bf b}_{(x,u)} = \begin{bmatrix}
           c(x,u,\mu_1) \\
           c(x,u,\mu_2) \\
           \vdots \\
           c(x,u,\mu_{M_x})
         \end{bmatrix}, \quad {\bf d}_{(x,u)}= \begin{bmatrix}
           \mathcal{T}(x^1|x,u,\mu_1) &  \mathcal{T}(x^2|x,u,\mu_1) & \dots & \mathcal{T}(x^{|\mathds{X}|}|x,u,\mu_1)  \\
           \mathcal{T}(x^1|x,u,\mu_2) &  \mathcal{T}(x^2|x,u,\mu_2) & \dots & \mathcal{T}(x^{|\mathds{X}|}|x,u,\mu_2)\\
           \vdots \\
           \mathcal{T}(x^1|x,u,\mu_{M_x}) &  \mathcal{T}(x^2|x,u,\mu_{M_x}) & \dots & \mathcal{T}(x^{|\mathds{X}|}|x,u,\mu_{M_x})
         \end{bmatrix} .
\end{align}\label{lin_fit2}
Furthermore, we also define ${\bf A}_{(x,u)}\in \mathds{R}^{d\times M_x}$
\begin{align}
{\bf A}_{(x,u)} = \left[ {\bf \Phi}_{(x,u)}(\mu_1), \dots, {\bf \Phi}_{(x,u)}(\mu_{M_x})\right].
\end{align}
Assuming that ${\bf A}_{(x,u)}$ has linearly independent columns, i.e. ${\bf \Phi}_{(x,u)}(\mu_i)$ and ${\bf \Phi}_{(x,u)}(\mu_j)$ are linearly independent for $\mu_i\neq\mu_j$, the estimates for $\theta_{(x,u)}$ and ${\bf Q}_{(x,u)} $ can be written as follows
\begin{align}\label{coor_learn}
\theta_{(x,u)} &= \left({\bf A}^\intercal_{(x,u)} {\bf A}_{(x,u)}\right)^{-1} {\bf A}^\intercal_{(x,u)} {\bf b}_{(x,u)}\nonumber\\
{\bf Q}_{(x,u)}& =  \left({\bf A}^\intercal_{(x,u)} {\bf A}_{(x,u)}\right)^{-1} {\bf A}^\intercal_{(x,u)} {\bf d}_{(x,u)}.
\end{align}
Note that above, each row of ${\bf Q}_{(x,u)}$ represents a signed measure on $\mathds{X}$.


\subsection{Independent Learning with Linear Function Approximation for Infinitely Many Players}\label{ind_learn_inf}
In this section, we will introduce a learning method where the agents perform independent learning to some extent. Here, rather than using a training set, we will focus on an online learning algorithm where at every time step, agent $i$ observes $x^i,u^i, X_1^i,c(x^i,u^i,\mu),\mu$. That is, each agent has access to their local state, action, cost realizations, one-step ahead state, and the mean-field term. However, they do not have access to local information about the other agents.

We first argue that full decentralization is usually not possible in the context of learning either. Recall that the mean-field flow is deterministic if every agent follows the same independently randomized agent-level policy. Furthermore, the flow of the mean-field terms remains deterministic even when the agents choose different exploration policies if the randomization is independent. To see this, assume that each agent picks some policy $\gamma_w(du|x)$ randomly by choosing $w\in \mathds{W}$ from some arbitrary distribution, where the mapping  $w\to \gamma_w(du|x)$ is predetermined. If the agents pick $w\in\mathds{W}$ independently, the mean-field dynamics is given by
\begin{align*}
\mu_{t+1} (\cdot)= \int \mathcal{T}(\cdot|x,u,\mu_t)\gamma_w(du|x)P_w(dw)\mu_t(dx) 
\end{align*}
where $P_w(dw)$ is the distribution by which the agents perform their independent randomization for the policy selection. Hence, the mean-field term dynamics follow a deterministic flow.  Note that for the above example, for simplicity, we assume that the agents pick according to the same distribution. In general, even if the agents follow different distributions for $w\in\mathds{W}$, the dynamics of the mean-field flow would remain deterministic according to a mixture distribution.  

Deterministic behavior might cause poor exploration performance. There might be cases where the mean-field flow gets stuck at a fixed distribution without learning or exploring the `important' parts of the space $\P(\mathds{X})$ sufficiently. To overcome this issue, and to make sure that the system is stirred sufficiently well during the exploration, one option is to introduce a common randomness for the selection of the exploration policies. In particular, each agent follows a randomized policy $\gamma_w(du|x)$   where the common randomness $w\in\mathds{W}$ is mutual information. Then the dynamics of the mean-field flow can be written as
\begin{align}\label{expo_random}
\mu_{t+1}(\cdot)=F(\mu_t,w):= \int \mathcal{T}(\cdot|x,u,\mu_t)\gamma_w(du|x)\mu_t(dx).
\end{align}
The common noise variable ensures  a level of coordination among the agents. However, this is still a significant relaxation compared to full coordination where the agents share their full state or control data.  In this section, we show that agents can construct independent learning iterates that converge by coordinating through an arbitrary common source of randomness.

We assume the following for the mean-field flow during the exploration:
\begin{assumption}\label{exp_assmp}
Consider the Markov chain $\{\mu_t\}_t\subset \P(\mathds{X})$ whose dynamics are given by (\ref{expo_random})
We assume that $\mu_t$ has geometric ergodicity with a unique invariant probability measure $P(\cdot)\in \P(\P(\mathds{X}))$ such that 
\begin{align*}
\|Pr(\mu_t\in\cdot) - P(\cdot)\|_{TV}\leq K \rho^t
\end{align*}
for some $K<\infty$ and some $\rho<1$.
\end{assumption}

\begin{remark}
We can establish some sufficient conditions on the transition kernel of the system to test the ergodicity. We note that (\cite[p 56, 3.1.1]{hernandez2012adaptive}) a sufficient condition is the following: there exists a mean-field state, say $\mu^*\in\P(\mathds{X})$ such that
\begin{align}\label{reach}
Pr\left( w:\int \mathcal{T}(\cdot|x,u,\mu)\gamma_w(du|x)\mu(dx)=\mu^*(\cdot)  \right)>0, \text{ for all } \mu \in\P(\mathds{X}).
\end{align}
That is, we need to be able to find a set of common noise realizations whose induced randomized exploration policies can take the state distribution to $\mu^*$ independent of the starting distribution $\mu$.

The assumption stated in this form indicates that the condition is of the stochastic reachability or controllability type. It requires that from any initial distribution $\mu$, there exists a control policy that can steer the distribution of the system to some target measure $\mu^*$. We also note that the above can be generalized to a $k$-step transition requirement. Analyzing this stochastic controllability behavior for the mean field systems is beyond the scope of the current paper, however, we give some examples in what follows.

We look further into (\ref{reach}). We fix some $x_1\in\mathds{X}$, and we define the following values:
\begin{align*}
&U(x,\mu) = \max_u\mathcal{T}(x_1|x,u,\mu)\\
&L(x,\mu) = \min_u\mathcal{T}(x_1|x,u,\mu).
\end{align*}
Note that $\mu^*(x_1)$ is the average of the probabilities $\mathcal{T}^\gamma(x_1|x,\mu):=\int\mathcal{T}(x_1|x,u,\mu)\gamma_w(du|x)$ under the measure $\mu$ for the $x$ values. Furthermore, by selecting an appropriate randomized policy, we can control these values in the interval $$I(x,\mu):=\left[L(x,\mu),U(x,\mu)\right].$$ Thus, if the intersection of these intervals is nonempty, i.e. if
\begin{align}\label{common_prob}
\bigcap_{x,\mu}I(x,\mu)\neq\emptyset
\end{align}
then one can set $\mu^*(x_1)$ to be a value in this intersection independent of $\mu$. By doing this for all $x_1$, we can set a reachable $\mu^*$ from any $\mu\in\P(\mathds{X})$. As a result, if (\ref{common_prob}) holds for every $x_1$, then (\ref{reach}), and thus Assumption \ref{exp_assmp} can be shown to hold.

A somewhat restrictive example for (\ref{common_prob}) is the following: assume that there exists a control action $u^*$ which can reset the state to some $x^*$ from any state $x$ and any mean-field term $\mu$, that is 
\begin{align*}
\mathcal{T}(x^*|x,u^*,\mu)=1 \text{ for all } x,\mu.
\end{align*}
This means that $1\in I(x,\mu)$ for all $x,\mu$ when $x_1=x^*$, and $0\in I(x,\mu)$ for all $x,\mu$ when $x_1\neq x^*$. Hence, $\mu_1(\cdot)= \delta_{x^*}(\cdot)$ can be reached from any starting point by applying the policy $\gamma(x)=u^*$ for all $x$. If this policy is among the set of exploration policies, the ergodicity assumption for the mean-field flow would be satisfied.  

We note again that a general result would require more in depth analysis, however, the above gives some idea on the implications of this assumption on the controllability of the mean-field model.
\end{remark}

We now define the trained measures, $P_{x} \in \P(\P(\mathds{X}))$ for each $(x,u)$ pair based on the invariant measure for the mean-field flow. Under Assumption \ref{pos_expo}, that is assuming $Pr(u|x)=\int \gamma_w(u|x)P_w(dw)>0$ for all $(x,u)$ pairs, we can write
\begin{align}\label{learned_set_ind}
    P(\mu\in A|x,u)&=\frac{Pr(x,u,\mu\in A)}{Pr(x,u)}\nonumber\\
    &=\frac{\int_{\mu\in A} \int_w \gamma_w(u|x)P_w(dw)\mu(x) P(d\mu) }{\int_{\mu\in \P(\mathds{X})} \int_w \gamma_w(u|x)P_w(dw)\mu(x) P(d\mu) }\nonumber\\
    &=\frac{\int_{\mu\in A} \mu(x) P(d\mu) }{\int_{\mu\in \P(\mathds{X})} \mu(x) P(d\mu) }=:P_x(\mu\in A)
\end{align}
Note that the trained sets of mean-field terms are independent of the control action $u$, as the exploration policies are independent of the mean-field terms given the state $x$. These sets have similar implications as the sets defined in (\ref{learned_set}). In particular, they indicate for which mean-field terms, one can estimate the kernel $\mathcal{T}(\cdot|x,u,\mu)$ and the cost function $c(x,u,\mu)$ via the training process. 

We now summarize the algorithm used for each agent.  We drop the dependence on agent identity $i$, and summarize the steps for a generic agent. At every time step $t$, every agent performs the following steps:
\begin{itemize}
\item Observe the common randomness $w$ given by the coordinator, and pick an action such that $u_t\sim \gamma_w(\cdot|x_t)$
\item Collect $x_t,u_t,x_{t+1},\mu_t,c$ where $c=c(x_t,u_t,\mu_t)$
\item For all $(x,u)\in\mathds{X}\times\mathds{U}$
\begin{align}\label{ind_learn1}
{\bf \theta}_{t+1}(x,u) = {\bf \theta}_t(x,u) + \alpha_t(x,u){\bf \Phi}(x,u,\mu_t)\left[c-{\bf \Phi}^\intercal(x,u,\mu_t) {\bf \theta}_t(x,u)\right]
\end{align}
\item Note that the signed measure vector ${\bf Q}_t(\cdot|x,u)$ consists $d$ signed measures defined on $\mathds{X}$, we denote by ${\bf Q}^j_t(x,u)$ the vector values of  ${\bf Q}_t(x^j|x,u)$ for $x^j\in \mathds{X}$ where $j\in\{1\dots,|\mathds{X}|\}$. For all $(x,u)$ and $j\in\{1\dots,|\mathds{X}|\}$
\begin{align}\label{ind_learn2}
{\bf Q}_{t+1}^j(x,u) = {\bf Q}_t^j(x,u) + \alpha_t(x,u) {\bf \Phi}(x,u,\mu_t)\left[\mathds{1}_{\{x_{t+1}=x^j\}} - {\bf \Phi}^\intercal(x,u,\mu_t) {\bf Q}_t^j(x,u) \right].
\end{align}

\end{itemize}

We next show that the above algorithm converges if the learning rates are chosen properly. To show the convergence, we first present a convergence result for stochastic gradient descent algorithms with quadratic cost, where the error is Markov and stationary.  We note that similar results have been established in the literature for the stochastic gradient iterations under Markovian noise processes under various assumptions; however, verifying most of these assumptions, such as the boundedness of the gradient, the boundedness of the iterates, or uniformly bounded variance, is not straightforward. Hence, we provide a proof in the appendix for completeness.
\begin{prop}\label{sgd} 
Let $\{S_t\} \subset \mathds{S}$ denote a Markov chain with the invariant probability measure $\pi(\cdot)$ where $\mathds{S}$ is a standard Borel space.  We assume that $\{S_t\}$ has geometric ergodicity such that $\|Pr(S_t\in\cdot) - \pi(\cdot)\|_{TV}\leq K\rho^t$ for some $K<\infty$ and some $\rho<1$.  Let $g(s,v)$ be  such that 
\begin{align*}
g(s,v) = \left( k(s)^\intercal v - h(s)  \right)^2
\end{align*}
for some $k:\mathds{S}\to \mathds{R}^d$, $h:\mathds{S}\to \mathds{R}$  and for $v\in\mathds{R}^d$. We assume that $k,h$ are uniformly bounded.  We denote by 
\begin{align*}
&G_t(v)= E[g(S_t,v)]\\
&G(v) = \int g(s,v) \pi(ds).
\end{align*}
Consider the iterations
\begin{align*}
v_{t+1} = v_t - \alpha_t \nabla g(S_t,v_t)
\end{align*}
where the gradient is with respect to $v_t$. If the learning rates are such that $\sum_t\alpha_t=\infty$ and $\sum_t\alpha_t^2<\infty$ with probability one,  we then have that $G_t(v_t) \to \min_vG(v) = G(v^*)$ almost surely.
\end{prop}
\begin{proof}
The proof can be found in Appendix \ref{sgd_proof}.
\end{proof}

\begin{corollary}\label{sgd_inf}
Let Assumption \ref{exp_assmp} and Assumption \ref{pos_expo} hold and let the learning rates be chosen such that $\alpha_t(x,u)=0$ unless $(X_t,U_t)=(x,u)$. Furthermore, $\sum_t \alpha_t(x,u)=\infty$ and $\sum_t \alpha_t^2(x,u)<\infty$ with probability one for all $(x,u)\in\mathds{X}\times\mathds{U}$.
Then, the iterations given in (\ref{ind_learn1}) and (\ref{ind_learn2}) converge with probability 1. Furthermore, the limit points, say ${\bf \theta}^*(x,u)$ and ${\bf Q}^*_{(x,u)}(\cdot)$ are such that 
\begin{align*}
{\bf \theta}^*(x,u)= \argmin_{\theta(x,u)\in \mathds{R}^d} \int \left| c(x,u,\mu) - {\bf \Phi}^\intercal(x,u,\mu) {\bf \theta}(x,u)\right|^2 P_{x}(d\mu)\\
{\bf Q}^{*,j}{(x,u)} = \argmin_{{\bf Q}^j{(x,u)} \in \mathds{R}^{d}}  \int \left|\mathcal{T}(j|x,u,\mu) - {\bf \Phi}_{x,u}^\intercal(\mu)   {\bf Q}^j{(x,u)} \right|^2 P_{x}(d\mu)
\end{align*}
for every $(x,u)$ pair and for every $j$, where ${\bf Q}^{j,*}{(x,u)}$ is the $j$th column of  ${\bf Q}^*_{(x,u)}(\cdot)$. Furthermore, $P_{x}(\cdot)$ denotes the trained set based on the invariant measure of the mean-field flow under the exploration policy with common randomness (see (\ref{learned_set_ind})).
\end{corollary}

\begin{proof}
We define the following stopping times 
\begin{align*}
    \tau_{k+1} = \min{\{t> \tau_k: (X_t,U_t)=(x,u) \}}
\end{align*}
such that $\tau_k$ indicates the $k$-th time the $(x,u)$ pair is visited.

For the iterations (\ref{ind_learn1}), Proposition \ref{sgd} applies such that for each $(x,u)$,  $v_k\equiv \theta_{\tau_k}(x,u)$, $k(\mu) \equiv {\bf \Phi}^\intercal(x,u,\mu)$ and $h(\mu) \equiv c(x,u,\mu)$ and finally the noise process $s_k\equiv \mu_{\tau_k}$. Note that ${\bf \Phi}$ and $c$ are assumed to be uniformly bounded which also agrees with the assumptions in Propositions \ref{sgd}. Furthermore, with the strong Markov property, $\mu_{\tau_k}$ is also a Markov chain which is sampled when the state-action pair is $(x,u)$. Thus, the invariant measure for the sampled process is $P_{x}$  as defined in (\ref{learned_set_ind}).

For the iterations (\ref{ind_learn2}),  Proposition \ref{sgd} applies such that for each $(x,u)$ and each $j$,  $v_t\equiv {\bf Q}^j_{\tau_k}(x,u)$, $k(\mu) \equiv {\bf \Phi}^\intercal(x,u,\mu)$ and $h(\mu) \equiv \mathds{1}_{\{X_{1} =x^j  \}}$. We note that $X_{1} = f(x,u,\mu,w)$ (see (\ref{dynamics})) where $w$ is the i.i.d. noise for the dynamics of agents. Thus, the noise process for iterations (\ref{ind_learn2}) can be taken to be the joint process $(\mu_t,w_t)$ where $\mu_t$ is an ergodic Markov process, and $w_t$ is an i.i.d. process. In particular, for every $(x,u)$ pair and for every $x^j$, if we consider the expectation over $(\mu,w)$ where $\mu\sim P_{x}(\cdot)$,  we get
\begin{align*}
E\left[ \mathds{1}_{\{X_{1} =x^j  \}}\right] = E\left[ \mathds{1}_{\{f(x,u,\mu,w) =x^j  \}}\right] =\int  \mathcal{T}(x^j|x,u,\mu)P_{x}(d\mu)
\end{align*}
for every $x^j\in \mathds{X}$.

The algorithm in (\ref{ind_learn2}) minimizes 
\begin{align*}
 \int \left|\mathds{1}_{\{f(x,u,\mu,w) =x^j  \}}- {\bf \Phi}_{x,u}^\intercal(\mu)   {\bf Q}^j ({x,u})  \right|^2 P_{x}(d\mu)P_w(dw)
\end{align*}
for each $j$ where $P_w$ is the distribution of the noise term. We can then open up the above term to write:
\begin{align*}
     &\argmin_{{\bf Q}^j(x,u)}\int \left|\mathds{1}_{\{f(x,u,\mu,w) =x^j  \}}- {\bf \Phi}_{x,u}^\intercal(\mu)   {\bf Q}^j({x,u}) \right|^2 P_{x}(d\mu)P_w(dw)\\
       &=\argmin_{{\bf Q}^j(x,u)}\int (\mathds{1}_{\{f(x,u,\mu,w) =x^j  \}})^2 - 2\mathds{1}_{\{f(x,u,\mu,w) =x^j  \}} {\bf \Phi}_{x,u}^\intercal(\mu)   {\bf Q}^j({x,u})  \\
       &\qquad\qquad \qquad + ({\bf \Phi}_{x,u}^\intercal(\mu)   {\bf Q}^j({x,u}) )^2 P_{x}(d\mu)P_w(dw)\\
       &=\argmin_{{\bf Q}^j(x,u)}\int  - 2\mathds{1}_{\{f(x,u,\mu,w) =x^j  \}} {\bf \Phi}_{x,u}^\intercal(\mu)   {\bf Q}^j({x,u}) + ({\bf \Phi}_{x,u}^\intercal(\mu)   {\bf Q}^j({x,u}) )^2 P_{x}(d\mu)P_w(dw)\\
        &=\argmin_{{\bf Q}^j(x,u)}\int  - 2  \mathcal{T}(x^j|x,u,\mu) {\bf \Phi}_{x,u}^\intercal(\mu)   {\bf Q}^j({x,u}) + ({\bf \Phi}_{x,u}^\intercal(\mu)   {\bf Q}^j({x,u}) )^2 P_{x}(d\mu)\\
        & =\argmin_{{\bf Q}^j(x,u)}\int  (\mathcal{T}(x^j|x,u,\mu))^2   - 2  \mathcal{T}(x^j|x,u,\mu) {\bf \Phi}_{x,u}^\intercal(\mu)   {\bf Q}^j({x,u}) + ({\bf \Phi}_{x,u}^\intercal(\mu)   {\bf Q}^j({x,u}) )^2 P_{x}(d\mu)\\
        &=\argmin_{{\bf Q}^j(x,u)}\int  \left(\mathcal{T}(x^j|x,u,\mu)) -  {\bf \Phi}_{x,u}^\intercal(\mu)   {\bf Q}^j({x,u})  \right)^2P_{x}(d\mu).
\end{align*}
Hence, the algorithm minimizes 
\begin{align*}
  \int  \left(\mathcal{T}(x^j|x,u,\mu)) -  {\bf \Phi}_{x,u}^\intercal(\mu)   {\bf Q}^j({x,u})  \right)^2P_{x}(d\mu)
\end{align*}
for each $j$.
\end{proof}

\subsection{Learning for Finitely Many Players}\label{ind_learn_fin}  In this section, we will study the more realistic scenario in which the number of agents is large but finite. The learning methods presented in the previous sections have focused on the ideal case where the system has infinitely many players. Although the setting with the infinitely many agents helps us to fix the ideas for the learning in the mean-field control setup, we should note that it is only an artificial setup, and the infinite population setup is only used as an approximation for large population control problems. Hence, we need to study the  actual setup for which the limit problem is argued to be a well approximation, that is the problem with very large but finitely many agents. 

We will apply the independent learning algorithm presented for the infinite population case, and study the performance of the learned solutions for the finitely many player setting. In particular, we will assume that the agents follow the iterations given in (\ref{ind_learn1}) and (\ref{ind_learn2}). We note, however, that the agents will not need to use common randomness during exploration as the flow of the mean-field term is stochastic for finite populations without common randomness. The method remains valid under common randomness as well; in fact, the common randomness, in general, encourages the exploration of the state space. The method is identical to the one presented in Section \ref{ind_learn_inf}. However, we present the method again since it has some subtle differences. 

At every time step $t$, agent $i$ performs the following steps:
\begin{itemize}
\item Pick an action such that $u^i_t\sim \gamma^i(\cdot|x^i_t)$
\item Collect $x^i_t,u^i_t,x^i_{t+1},\mu^N_t ,c$ where $c=c(x^i_t,u^i_t,\mu^N_t)$ and $\mu_t^N=\mu_{\bf x_t}$
\item For all $(x,u)\in\mathds{X}\times\mathds{U}$
\begin{align}\label{finite_ind1}
{\bf \theta}_{t+1}(x,u) = {\bf \theta}_t(x,u) + \alpha_t(x,u){\bf \Phi}_{x,u}(\mu^N_t)\left[c-{\bf \Phi}^\intercal_{x,u}(\mu^N_t) {\bf \theta}_t(x,u)\right]
\end{align}
\item  Denoting by ${\bf Q}^j_t(x,u)$ the vector values of  ${\bf Q}_t(x^j|x,u)$ for all $x^j\in \mathds{X}$ where $j\in\{1\dots,|\mathds{X}|\}$. For all $(x,u)$ and $j\in\{1\dots,|\mathds{X}|\}$
\begin{align}\label{finite_ind2}
{\bf Q}_{t+1}^j(x,u) = {\bf Q}_t^j(x,u) + \alpha_t(x,u) {\bf \Phi}_{x,u}(\mu^N_t)\left[\mathds{1}_{\{x^i_{t+1}=x^j\}} - {\bf \Phi}_{x,u}^\intercal(\mu^N_t) {\bf Q}_t^j(x,u) \right]
\end{align}
\end{itemize}

\begin{remark}
We note that the iterates $\theta_t$ and ${\bf Q}_t$ depend on the agent identity $i$, in this case, as each agent can learn the model independently. Moreover, the learning rates $\alpha_t(x,u)$ and the basis functions ${\bf\Phi}_{x,u}$ might depend on the agent identity as well. However, we omit the dependence in the notation to reduce notational clutter. 
\end{remark}

\begin{assumption}\label{finite_erg}
Under the exploration team policy ${\bf \gamma}(\cdot|{\bf x}) =\left[ \gamma^1(\cdot|x^1),\dots , \gamma^N(\cdot|x^N)\right]^\intercal$,  the state vector process ${\bf x_t}=[x_t^1,\dots,x_t^N]$  of the agents is irreducible and aperiodic and in particular admits a unique invariant measure, and thus the mean-field flow $\mu^N_{t}= \mu_{\bf x_t}$ admits a unique invariant measure, say $P^N(\cdot)\in \P_N(\mathds{X})$, as well.
\end{assumption}
\begin{remark}
We note that a sufficient condition for the above assumption to hold is that there exists some $x'\in\mathds{X}$ such that $\mathcal{T}(x'|x,u,\mu^N)\geq\epsilon>0$ for any $x,u$ and for any $\mu^N$. In particular, this implies that 
\begin{align*}
Pr\left({\bf X}_{t+1}=[x',\dots,x']|{\bf x}_t,\gamma  \right)=\prod_{i=1}^N \sum_u\mathcal{T}(x'|x_t^i,u,\mu_t^N)\gamma^i(u|x^i_t)\geq \epsilon^N>0
\end{align*}
and thus \cite[p 56, 3.1.1]{hernandez2012adaptive} implies that the process ${\bf X}_t$ is geometrically ergodic.
\end{remark}

The next result shows the convergence of the algorithm. Similar to the previous section, we first define the trained sets of mean-field terms for every $(x,u)$ pair using the stationary distribution of the mean-field terms. We assume that Assumption \ref{pos_expo} holds for every policy $\gamma^i$ such that $\gamma^i(u|x)>0$ for every $(x,u)$ pair. Denoting the invariant distribution of the joint state process by $P(\bf x)$, for some  $\mu^N\in\P_N(\mathds{X})$, for agent $i$, we can write
\begin{align}\label{learned_model_fin}
     P(\mu^N|(x^i,u^i)=(x,u)) & =\int_{{\bf x}\in\mathds{X}^N}Pr(\mu^N|{\bf x})P({\bf x}|(x^i,u^i)=(x,u))\nonumber\\
    &= \int_{{\bf x}\in\mathds{X}^N} \mathds{1}_{\{\mu^N=\mu_{\bf x} \}} \frac{\gamma^i(u|x)\mathds{1}_{\{ x={\bf x}[i] \}}   }{P(x^i=x,u^i=u)} P({\bf x})\nonumber\\
    &= \int_{{\bf x}\in\mathds{X}^N} \mathds{1}_{\{\mu^N=\mu_{\bf x} \}} \frac{\gamma^i(u|x)\mathds{1}_{\{ x={\bf x}[i] \}} }{\gamma^i(u|x)P(x^i=x) } P({\bf x})\nonumber\\
    & =\int_{{\bf x}\in\mathds{X}^N} \mathds{1}_{\{\mu^N=\mu_{\bf x} \}} \frac{\mathds{1}_{\{ x={\bf x}[i] \}} }{P(x^i=x) } P({\bf x})\nonumber\\
    &=P(\mu^N|x^i=x)=:P^{i}_x(\mu^N).
\end{align}
Note that the trained measure of mean-fields is independent of the control actions; however, it does depend on the agent identity as the agents follow distinct exploration policies.

\begin{corollary}
Let Assumption \ref{finite_erg} hold, and let Assumption \ref{pos_expo} hold for each policy $\gamma^i$. Assume further that the learning rates of every agent satisfy the assumption of Corollary \ref{sgd_inf} such that $\sum_t \alpha_t(x,u)=\infty$ and $\sum_t \alpha_t^2(x,u)<\infty$ with probability one for all $(x,u)\in\mathds{X}\times\mathds{U}$. Then, the iterations given in (\ref{finite_ind1}) and (\ref{finite_ind2}) converge with probability one. Furthermore, for agent $i$, the limit points, say ${\bf \theta}^*_{(x,u)}$ and ${\bf Q}^*_{(x,u)}(\cdot)$ are such that 
\begin{align*}
{\bf \theta}^*(x,u)= \argmin_{\theta(x,u)\in \mathds{R}^d} \int \left| c(x,u,\mu) - {\bf \Phi}^\intercal(x,u,\mu) {\bf \theta}(x,u)\right|^2 P^i_{x}(d\mu)\\
{\bf Q}^{*,j}{(x,u)} = \argmin_{{\bf Q}^j{(x,u)} \in \mathds{R}^{d}}  \int \left|\mathcal{T}(j|x,u,\mu) - {\bf \Phi}_{x,u}^\intercal(\mu)   {\bf Q}^j{(x,u)} \right|^2 P^i_{x}(d\mu)
\end{align*}
for every $(x,u)$ pair and for every $j$, where ${\bf Q}^{j,*}{(x,u)}$ is the $j$th column of  ${\bf Q}^*_{(x,u)}(\cdot)$. Furthermore, $P_{x}(\cdot)$ denotes the trained set based on the invariant measure of the mean-field flow under the exploration policy with common randomness (see (\ref{learned_model_fin})).
\end{corollary}

\begin{proof}
The proof is identical to the proof of Corollary \ref{sgd_inf}, and is an application of Proposition \ref{sgd}. The only difference is the ergodicity of the mean-field process $\mu_t^N$, which does not require common randomness for exploration policies.
\end{proof}

\section{Uniform Error Bounds for Model Approximation}\label{unif_err_est}
The learning methods we have presented in Section \ref{learning_sec} minimize the $L_2$ distance between the true model and the linear approximate model, under the probability measure induced by the training data. In particular, denoting the learned parameters for a fixed pair $(x,u)\in\mathds{X}\times\mathds{U}$ by $\theta^*_{(x,u)}$, and ${\bf Q}_{(x,u)}^*(\cdot)$, we have that
 \begin{align}\label{lin_least_square}
&{\bf \theta}^*_{(x,u)}= \argmin_{\theta\in\mathds{R}^d} \int \left| c(x,u,\mu) - {\bf \Phi}_{(x,u)}^\intercal(\mu) {\bf \theta}_{(x,u)}\right|^2 P_{x}(d\mu)\nonumber\\
&{\bf Q}^{*,j}{(x,u)} = \argmin_{{\bf Q}^j{(x,u)} \in \mathds{R}^{d}}  \int \left|\mathcal{T}(j|x,u,\mu) - {\bf \Phi}_{(x,u)}^\intercal(\mu)   {\bf Q}^j{(x,u)} \right|^2 P_{x}(d\mu)
\end{align}
for some probability measure $P_{x}(\cdot)\in \P(\mathds{X})$. The measure $P_{x}(\cdot)$ depends on the learning method used. 

\begin{itemize}
\item  For the coordinated learning methods presented in Section \ref{coor_learn_sec}, $P_{x}(\cdot)$ represents the empirical distribution of the mean-field terms in the training data for which the state $x$ has positive measure (see (\ref{learned_set})).
 \item For the individual learning method presented in Section \ref{ind_learn_inf} for infinite populations, $P_{x}(\cdot)$ represents the invariant measure of the mean-field flow under the randomized exploration policies given the state $x$ is observed. See (\ref{learned_set_ind}).
\item Finally, for the individual learning method for finite populations in Section \ref{ind_learn_fin}, $P_{x}$ depends on the agent identity $i$, and thus denoted by $P^i_x$. Similar to the infinite population setting, it represents the invariant measure of the process $\mu_{\bf x_t}$ conditioned on the event $(x^i=x)$,  for agent $i$ where ${\bf x_t}$ is the $N$ dimensional vector state of the team of $N$ agents. We note that each agent might have different trained sets of mean-field terms in this setting, since the policies may be distinct. 
\end{itemize} 

 When the learned policy is executed, the flow of the mean-field is not guaranteed to stay in the support of the training measure $P_{x}(\cdot)$.  Hence, in what follows, we aim to generalize the $L_2$ performance of the learned models over the space $\P(\mathds{X})$.

In what follows, we will sometimes refer to $P_{x}(\cdot)$ as the {\it training measure}.

\subsection{Ideal Case: Perfectly Linear Model}\label{lin_model} If the cost and the kernel are fully linear for a given set of basis functions ${\bf \Phi}_{(x,u)}(\mu) = [\Phi^1_{(x,u)}(\mu),\dots \Phi^d_{(x,u)}(\mu)]^\intercal$ then the linear model can be learned perfectly. That is, for the given basis functions ${\bf \Phi}_{(x,u)}(\mu) $, there exist ${\bf \theta}^*_{(x,u)}$ and ${\bf Q}^*_{(x,u)}(\cdot)$ such that 
\begin{align*}
&c(x,u,\mu) = {\bf \Phi}_{(x,u)}^\intercal(\mu)\theta^*_{(x,u)}\\
&\mathcal{T}(\cdot|x,u,\mu)= {\bf \Phi}_{(x,u)}^\intercal(\mu) {\bf Q}^*_{(x,u)}(\cdot)
\end{align*}
The model can be learned perfectly with a coordinator under the method presented in Section \ref{coor_learn_sec} if
\begin{itemize}
\item the training set $T$ is such that for each pair $(x,u)$, there exist at least $d$ different data points. Furthermore, for a given data point of the form $(x^i,u^i,X_1^i,\mu,c^i)_{i=1}^\infty$, the  state-action distribution for this point is such that $Pr(x,u)>0$
\item and if the basis functions ${\bf \Phi}_{(x,u)}(\mu)$ and  ${\bf \Phi}_{(x,u)}(\mu')$ are  linearly independent for every $\mu\neq\mu'$ that is if ${\bf A}_{x,u}$ (see (\ref{lin_fit2})) has independent columns.
\end{itemize}

For the independent learning methods given in Section \ref{ind_learn_inf} and Section \ref{ind_learn_fin}, the learned model will be the true model with no error if the iterations converge.

\subsection{Nearly Linear Models}
In this section, we provide a result that states that if the true model can be approximated $\epsilon$ close to a linear model, then the models learned with the least square method can approximate the true model uniformly in the order of $\epsilon$ if the training set is informative enough.

The following assumption states that the true model is nearly linear.
\begin{assumption}\label{mis_ass}
We assume the existence of $\bar{\theta}_{x,u}\in \mathds{R}^d$ and $\bar{\bf Q}_{x,u}(\cdot) \in \mathds{R}^{d\times |\mathds{X}|}$ with the following property: denoting by $\bar{\bf Q}^j(x,u) \in\mathds{R}^d $ the $j$th column of  $\bar{\bf Q}_{x,u}(\cdot)$, for some $\epsilon>0$, and $\epsilon_j>0$ 
\begin{align}\label{lin_mis}
&\sup_{x,u,\mu}\left|\mathcal{T}(j|x,u,\mu) - {\bf \Phi}_{x,u}^\intercal(\mu)   \bar{\bf Q}^j{(x,u)}  \right|\leq \epsilon_j\nonumber\\
&\sup_{x,u,\mu}\left| c(x,u,\mu)- {\bf \Phi}_{x,u}^\intercal(\mu)   \bar{\bf \theta}_{x,u}(\cdot)  \right|\leq \epsilon.
\end{align}
In particular, further assuming  $\sum_j \epsilon_j \leq \epsilon$, the above implies that
\begin{align*}
&\left\|\mathcal{T}(\cdot|x,u,\mu) - {\bf \Phi}_{x,u}^\intercal(\mu)   \bar{\bf Q}_{(x,u)}(\cdot)  \right\|\leq \frac{\epsilon}{2}
\end{align*}
for all $x,u,\mu$.
\end{assumption}



We note that, in general, there is no guarantee that the learned dynamics constitute a proper stochastic kernel. This can be guaranteed when the model is fully linear as discussed in Section \ref{lin_model} or when we consider a discretization based approximation as described in Section \ref{disc_model}. However, for general linear approximations, as in this section, we project the learned model ${\bf \Phi}_{(x,u)}^\intercal(\mu){\bf Q}^*_{(x,u)}$ onto the set of probability measures $P(\mathds{X})$, i.e. the simplex over $\mathds{X}$.

In particular, we use the following notation:
\begin{align}\label{est_model}
    &\hat{c}(x,u,\mu) := {\bf \Phi}^\intercal_{x,u}(\mu)\theta^*_{(x,u)} \nonumber\\
    &\hat{\mathcal{T}}(\cdot|x,u,\mu):=\argmin_{\mu\in\P(\mathds{X})}\|\mu-{\bf \Phi}^\intercal_{x,u}(\mu){\bf Q}^*_{(x,u)}(\cdot) \|
\end{align}
where $\theta^*_{(x,u)}$ and ${\bf Q}^*_{(x,u)}(\cdot)$ denote the learned models based on the least square method, see (\ref{lin_least_square}).

\begin{prop}
Let $P_{x}(\cdot) \subset \P(\P(\mathds{X}))$ denote the training distribution of the mean-field terms for the state $x\in\mathds{X}$. Let Assumption \ref{mis_ass} hold.  For the estimate models, $\hat{c}(x,u,\mu)$ and $\hat{\mathcal{T}}(\cdot|x,u,\mu)$ defined based on the least square method (see (\ref{est_model})), we have that
\begin{align*}
&\left| c(x,u,\mu) - \hat{c}(x,u,\mu)  \right| \leq   \epsilon\left(1+\frac{2\sqrt{d}}{\sqrt{\lambda_{\min}}}\right)\\
&\left\| \mathcal{T}(\cdot|x,u,\mu) - \hat{\mathcal{T}}(\cdot|x,u,\mu) \right\| \leq  \epsilon\left(1+\frac{2\sqrt{d}}{\sqrt{\lambda_{\min}}}\right)
\end{align*}
where $\lambda_{min}$ is the minimum eigenvalue of $\int {\bf \phi}_{(x,u)}(\mu){\bf \phi}_{(x,u)}^\intercal (\mu)P_x(d\mu)$.
\end{prop}

\begin{proof}
We first note that since the learned $\theta^*_{(x,u)}$ minimizes the $L_2$ distance to the true model under the training measure $P_{x}$,
(\ref{lin_mis}) implies that 
\begin{align*}
& \int_{\P(\mathds{X})} \left| c(x,u,\mu) - {\bf \Phi}_{(x,u)}^\intercal(\mu) {\bf \theta^*_{(x,u)}}\right|^2 P_{x}(d\mu) \leq & \int_{\P(\mathds{X})} \left| c(x,u,\mu) - {\bf \Phi}_{(x,u)}^\intercal(\mu) {\bf \bar{\theta}_{(x,u)}}\right|^2 P_{x}(d\mu)\leq \epsilon^2.
\end{align*}
In particular, via the triangle inequality (under the $L_2$ norm) we also have that
\begin{align*}
 &\int_{\P(\mathds{X})} \left|  {\bf \Phi}_{(x,u)}^\intercal(\mu) {\bf \bar{\theta}}_{(x,u)} -  {\bf \Phi}_{(x,u)}^\intercal(\mu) {\bf \theta^*_{(x,u)}}\right|^2 P_{x}(d\mu) \leq 4 \epsilon^2.
\end{align*}
We can further write that 
\begin{align*}
 &\int_{\P(\mathds{X})} \left|  {\bf \Phi}_{(x,u)}^\intercal(\mu) {\bf \bar{\theta}}_{(x,u)} -  {\bf \Phi}_{(x,u)}^\intercal(\mu) {\bf \theta^*_{(x,u)}}\right|^2 P_{x}(d\mu) \\
 &= \int_{\P(\mathds{X})} \left|  {\bf \Phi}_{(x,u)}^\intercal(\mu) \left({\bf \bar{\theta}}_{(x,u)} - \theta^*_{(x,u)}\right) \right|^2 P_{x}(d\mu) \\
 &=  \left({\bf \bar{\theta}}_{(x,u)} - \theta^*_{(x,u)}\right)^\intercal  \int_{\P(\mathds{X})}   {\bf \Phi}_{(x,u)}(\mu) {\bf \Phi}^\intercal_{(x,u)}(\mu) P_{x}(d\mu)  \left({\bf \bar{\theta}}_{(x,u)} - \theta^*_{(x,u)}\right)\\
 &\geq \left\|{\bf \bar{\theta}}_{(x,u)} - \theta^*_{(x,u)}\right\|_2^2 \lambda_{\min}
\end{align*}
where $\lambda_{min}$ is the minimum eigenvalue of $\int {\bf \phi}_{(x,u)}(\mu){\bf \phi}_{(x,u)}^\intercal (\mu)P_x(d\mu)$. Thus, we have that 
\begin{align*}
\left\|{\bf \bar{\theta}}_{(x,u)} - \theta^*_{(x,u)}\right\|_2 \leq \frac{2\epsilon}{\sqrt{\lambda_{\min}}}.
\end{align*}
Finally, using the triangle inequality with the fact that $\hat{c}(x,u,\mu) = {\bf \Phi}^\intercal_{x,u}(\mu)\theta^*_{(x,u)} $
\begin{align*}
&\left| c(x,u,\mu) - \hat{c}(x,u,\mu)  \right| \leq \left| c(x,u,\mu) - {\bf \Phi}_{(x,u)}^\intercal(\mu) {\bf \bar{\theta}}_{(x,u)} \right| + \left| {\bf \Phi}_{(x,u)}^\intercal(\mu) {\bf \bar{\theta}}_{(x,u)} - {\bf \Phi}^\intercal_{x,u}(\mu)\theta^*_{(x,u)}  \right|\\
&\leq \epsilon +  \|{\bf \Phi}_{(x,u)}(\mu) \|_2 \left\|{\bf \bar{\theta}}_{(x,u)} - \theta^*_{(x,u)}\right\|_2 \leq \epsilon + \frac{2\epsilon\sqrt{d}}{\sqrt{\lambda_{\min}}}
\end{align*}
where we used $\|{\bf \Phi}_{(x,u)}(\mu) \|_2 \leq \sqrt{d}$ since we assume that $
\|\Phi^j_{(x,u)}(\cdot)\|_\infty\leq 1$ via Assumption \ref{bdd_assmp}.

For the proof of the error bound of the estimate kernel $\hat{\mathcal{T}}(\cdot|x,u,\mu)$ we follow identical steps, recalling that by construction $  {\bf \Phi}_{(x,u)}^\intercal(\mu){\bf Q}^{*,j}{(x,u)}$ minimizes the $L_2$ distance to $\mathcal{T}(j|x,u,\mu)$ and using Assumption \ref{mis_ass}, we write that
\begin{align*}
 \lambda_{\min} \| {\bf \bar{Q}}^j{(x,u)} -  {\bf Q}^{*,j}{(x,u)}\|^2_2 \leq \int \left| {\bf \Phi^\intercal_{(x,u)}}(\mu) {\bf \bar{Q}}^j{(x,u)} -  {\bf \Phi}_{(x,u)}^\intercal(\mu) {\bf Q}^{*,j}{(x,u)}\right|^2 P_x(d\mu)\leq 4 \epsilon_j^2,
\end{align*}
which yields 
\begin{align*}
 \| {\bf \bar{Q}}^j{(x,u)} -  {\bf Q}^{*,j}{(x,u)}\|_2\leq \frac{2\epsilon_j}{\sqrt{\lambda_{\min}}}
\end{align*}
where $\lambda_{min}$ is the minimum eigenvalue of $\int {\bf \phi}_{(x,u)}(\mu){\bf \phi}_{(x,u)}^\intercal (\mu)P_x(d\mu)$. 

Since $\hat{\mathcal{T}}(\cdot|x,u,\mu)$ is the projection of ${\bf \Phi}_{x,u}^\intercal(\mu) {\bf Q}^*_{(x,u)}(\cdot)$ onto the space of probability measures, we have, by the definition of the projection, that
\begin{align*}
\|\hat{\mathcal{T}}(\cdot|x,u,\mu) - {\bf \Phi}_{x,u}^\intercal(\mu) {\bf Q}^*_{(x,u)}(\cdot)\|&\leq \|{\mathcal{T}}(\cdot|x,u,\mu) - {\bf \Phi}_{x,u}^\intercal(\mu) {\bf Q}^*_{(x,u)}(\cdot)\|.
\end{align*}
We then write the following using the triangle inequality
\begin{align*}
&\left\| \mathcal{T}(\cdot|x,u,\mu)-\hat{\mathcal{T}}(\cdot|x,u,\mu)\right\| \\ 
&\leq \left\| \mathcal{T}(\cdot|x,u,\mu)- {\bf \Phi}_{x,u}^\intercal(\mu) {\bf Q}^*_{(x,u)}(\cdot)\right\| + \left\|  {\bf \Phi}_{x,u}^\intercal(\mu) {\bf Q}^*_{(x,u)}(\cdot) - \hat{\mathcal{T}}(\cdot|x,u,\mu)\right\|\\
& \leq 2 \|{\mathcal{T}}(\cdot|x,u,\mu) - {\bf \Phi}_{x,u}^\intercal(\mu) {\bf Q}^*_{(x,u)}(\cdot)\|\\
&\leq 2 \left \|{\mathcal{T}}(\cdot|x,u,\mu) - {\bf \Phi}_{x,u}^\intercal(\mu) \bar{\bf Q}_{(x,u)}(\cdot) \right\| + 2 \left\|  {\bf \Phi}_{x,u}^\intercal(\mu) \bar{\bf Q}_{(x,u)}(\cdot) -  {\bf \Phi}_{x,u}^\intercal(\mu) {\bf Q}^*_{(x,u)}(\cdot)\right\|\\
&\leq \epsilon +  \sum_j \left|  {\bf \Phi}_{x,u}^\intercal(\mu) \bar{\bf Q}^j{(x,u)} -  {\bf \Phi}_{x,u}^\intercal(\mu) {\bf Q}^{*,j}{(x,u)}\right|\\
&\leq \epsilon +  \sum_j \| {\bf \Phi}_{x,u}^\intercal(\mu) \|_2  \| {\bf \bar{Q}}^j{(x,u)} -  {\bf Q}^{*,j}{(x,u)}\|_2\\
&\leq \epsilon +  \frac{\sqrt{d}}{\sqrt{\lambda_{\min}}} \sum_j {\epsilon_j} \leq \epsilon + 2 \frac{\sqrt{d} \epsilon}{\sqrt{\lambda_{\min}}}
\end{align*}
where we used $\sum_j\epsilon_j\leq \epsilon$ with Assumption \ref{mis_ass}.
\end{proof}

\subsection{A Special Case: Linear Approximation via Discretization}\label{disc_model}
 In this section, we show that the discretization of the space $\P(\mathds{X})$ can be seen as a particular case of linear function approximation with a special class of basis functions. In particular, for this case, we can analyze the error bounds of the learned policy with mild conditions on the model.

Let $\{B_i\}_{i=1}^d\subset \P(\mathds{X})$ be a disjoint set of quantization bins of $\P(\mathds{X})$ such that $\cup B_i = \P(\mathds{X})$. We define the basis functions for the linear approximation such that  
\begin{align*}
\Phi^i_{x,u}(\cdot) = \mathds{1}_{B_i}(\cdot)
\end{align*}
for all $(x,u)$ pairs. Note that in general the quantization bins, $B_i$'s, can be chosen differently for every $(x,u)$; for the simplicity of the analysis, we will work with a discretization scheme which is the same for every $(x,u)$. An important property of the discretization is that the basis functions form an orthonormal basis for any training measure $P(\cdot)$  with $P(B_i)>0$ for each quantization bin $B_i$. That  is 
\begin{align*}
\int \big\langle \Phi^i_{x,u}(\mu), \Phi^j_{x,u}(\mu)\big\rangle P(d\mu) = \mathds{1}_{\{i=j\}}
\end{align*}
for every $(x,u)$ pair. This property allows us to analyze the uniform error bounds of the discretization method more directly.

The linear fitted model (see (\ref{lin_least_square}) with the chosen basis functions becomes
\begin{align}\label{quant_coef}
\theta^i_{(x,u)} &=   \frac{\int_{B_i}c(x,u,\mu) P(d\mu)  }{P(B_i)}\nonumber\\
Q^i_{(x,u)}(\cdot) &= \frac{\int_{B_i} \mathcal{T}(\cdot|x,u,\mu)P(d\mu) }{P(B_i)} 
\end{align} 
In words, the learned coefficients are the averages of  cost and transition realizations from the training set of the corresponding quantization bin. 

The following then is an immediate result of Assumption \ref{main_assmp}.
\begin{prop}
Let $\theta_{x,u} = [\theta^1_{x,u} , \dots, \theta^d_{x,u}]^\intercal$ and ${\bf Q}_{x,u}(\cdot) = \left[ Q^1_{x,u}(\cdot) ,\dots, Q^d_{x,u}(\cdot)\right]^\intercal$ be given by (\ref{quant_coef}). If the training measure $P(\cdot)$ is such that $P(B_i)>0$ for each quantization bin $B_i$, under Assumption \ref{main_assmp}, we then have that
\begin{align*}
\left| c(x,u,\mu) - {\bf \Phi}^\intercal_{x,u}(\mu) {\bf \theta}_{x,u}\right| \leq K_c L\\
 \left\| \mathcal{T} (\cdot |x,u,\mu) - {\bf \Phi}^\intercal_{x,u}(\mu) {\bf Q}_{x,u}\right\| \leq K_f L
\end{align*}
where $L$ is the largest diameter of the quantizations bins such that
\begin{align*}
L = \max_i \sup_{\mu,\mu' \in B_i} \|\mu-\mu'\|.
\end{align*}
\end{prop}

\section{Error Analysis for Control with Misspecified Models}\label{perf_loss}
In the previous section, we have studied the uniform mismatch bounds of the learned models. 
In this section, we will focus on what happens if the controllers designed for the linear estimates are used for the true dynamics. We will provide error bounds for  the performance loss of the control designed for a possibly missepecified model.

We will analyze the infinite population and the finite population settings separately. We note that some of the following results (e.g. Lemma \ref{finite_inf_dif}) have been studied in the literature to establish the connection between the $N$-agent control problems and the limit mean-field control problem without the model mismatch aspect. That is, existing results study
what happens if one uses the infinite population solution for the finite population control problem with perfectly known dynamics (see e.g. \cite{motte2022quantitative,bauerle2021mean,motte2022mean}). However, we present the proof of every result for completeness and because of the connections in the analysis we follow throughout the paper. Furthermore, the existing results are often stated under slightly different assumptions and settings such as being stated only for closed loop policies, or only for policies that are open loop in the sense that they are measurable with respect to the noise process.

\subsection{Error Bounds for Infinitely Many Agents} 
As we have observed in Example \ref{dec_counter}, even when agents agree on the model knowledge, without coordination on which policy to follow, the optimality may not be achieved. Therefore, we assume that after the learning period, the team of agents collectively agrees on the cost and transition models given by $\hat{c}(x,u,\mu)$ and $\hat{\mathcal{T}}(\cdot|x,u,\mu)$ and
designs policies for this model. We will assume that 
\begin{align}\label{miss}
&\left|c(x,u,\mu)-\hat{c}(x,u,\mu)\right| \leq \lambda\nonumber\\
& \left\|\mathcal{T}(\cdot|x,u,\mu) - \hat{\mathcal{T}}(\cdot|x,u,\mu) \right\| \leq \lambda
\end{align}
for some $\lambda<\infty$ and for all $x,u,\mu$. That is $\lambda$ represents the uniform model mismatch constant.

We will consider two different cases for the execution of the designed control.
\begin{itemize}
\item {\it Closed loop control:} The team decides on a policy $\hat{g}:\P(\mathds{X}) \to \Gamma$, and uses their local states and the mean-field term to apply the policy $\hat{g}$. That is, an agent $i$ observes the mean-field term $\mu_t$, chooses $\hat{g}(\mu_t) = \hat{\gamma}(\cdot|x^i,\mu_t)$ and applies their control action according to $\hat{\gamma}(\cdot|x^i,\mu_t)$ with the local state $x^i$. The important distinction is that the mean-field term $\mu_t$ is observed by every agent, and they decide on their agent-level policies with the observed mean-field term. Hence, we refer to this execution method to be the closed loop method since the mean-field term is given as a feedback variable.
\item {\it Open loop control:} We have argued earlier that the flow of the mean-field term $\mu_t$ is deterministic for the infinitely many agent case, see (\ref{inf_kernel}). In particular, the mean-field term $\mu_t$ can be estimated with the model information. Hence, for this case, we will assume that the agents only observe their local states, and estimates the mean-field term independent instead of observing it. That is, an agent $i$ estimates the mean-field term $\hat{\mu}_t$, and applies their control action according to $\hat{\gamma}(\cdot|x^i,\hat{\mu}_t)$ with the local state $x^i$. Note that if the model dynamics were perfectly known, this estimate would coincide with the true flow of the mean-field term. However, when the model is misspecified, the estimate $\hat{\mu}_t$ and the correct mean-field term will deviate from each other, and we will need to study the effects of this deviation on the control performance, in addition to the incorrect computation of the control policy. 
\end{itemize}

In what follows, previously introduced constants $K_c, K_f$ and $\delta_T$ will be used often. We refer the reader to Assumption \ref{main_assmp} for $K_c,K_f$, and equation (\ref{dob}) for $\delta_T$.

For the results in this section, we will require that $\beta K<1$ where $K=K_f+\delta_T$. We note that this assumption is needed to show the Lipschitz continuity of the value function $K_\beta^*(\mu)$ with respect to $\mu$. The following provides 
an example where this bound is not satisfied, and the value function is not Lipschitz continuous.

\begin{example}
Consider a control-free (without loss of optimality) dynamics, with a binary state space $\mathds{X}=\{0,1\}$. We assume that 
\begin{align*}
\mathcal{T}(0|x,\mu)=\mu(0)^2
\end{align*}
that is, the state process moves to $0$ with probability $\mu(0)^2$ independent of the value of the state at the current step. We first notice that $\|\mu-\mu'\|=|\mu(0)-\mu'(0)|$ for the binary state space. Furthermore, we note that this kernel is Lipschitz continuous in $\mu$ with Lipschitz constant 2, that is $K_f=2$. To see this, consider the following for $\mu,\mu'\in\P(\mathds{X})$
\begin{align*}
&\|\mathcal{T}(\cdot|x,\mu)-\mathcal{T}(\cdot|x,\mu')\|=\frac{1}{2}\left(|\mathcal{T}(0|x,\mu)-\mathcal{T}(0|x,\mu')| +| \mathcal{T}(1|x,\mu)-\mathcal{T}(1|x,\mu')|  \right) \\
&=|\mu(0)^2-\mu'(0)^2|= |\mu(0)-\mu'(0)|\times |\mu(0)+\mu'(0)|\leq 2|\mu(0)-\mu'(0)|= 2\|\mu(\cdot)-\mu'(\cdot)\|
\end{align*}
where we used the bound that $ |\mu(0)+\mu'(0)|\leq 2$ which is the minimal uniform upper bound for all $\mu,\mu'$.

Hence, the kernel is Lipschitz continuous with constant 2. Furthermore, since the dynamics do not depend $x$ and $u$, we have that $\delta_T=0$, and thus $K=K_f+\delta_T=2$.

The stage-wise cost is given by $c(\mu)=\mu(0)$.
We consider Lipschitz continuity of the value function around $\mu(0)=1$ ,i.e. around $\mu=\delta_0$. Note that for some initial distribution $\mu_0(0)=a$, one can iteratively show that
\begin{align*}
\mu_t(0)=a^{{2^t}}.
\end{align*}
Hence, we can write the value function as 
\begin{align*}
K_\beta(a)=\sum_{t=0}^\infty \beta^t a^{{2^t}}.
\end{align*}
To show that this function is not Lipschitz continuous, we consider two points $a,b \in[0,1]$, without loss of generality assume that $a\geq b$:
\begin{align*}
\frac{K_\beta(a)-K_\beta(b)}{a-b}=\frac{\sum_{t=0}^\infty \beta^t (a^{2^t} - b^{2^t})}{a-b}=\sum_{t=0}^\infty \beta^t 2^t c^{2^t -1}
\end{align*}
for some $c\in[a,b]$ where we used the mean value theorem for $\frac{a^{2^t} - b^{2^t}}{a-b}$. We can see that the above cannot be bounded uniformly when $c$ is around $1$ if $\beta\geq 1/2$, i.e. if $\beta K \geq 1$. This implies that the value function cannot be Lipschitz continuous if $\beta K \geq 1$.
\end{example}

\subsubsection{Error Bounds for Closed Loop Control}
 We assume that the agents calculate an optimal policy, say $\hat{g}$, for the incorrect model ($\hat{\mathcal{T}}$ and $\hat{c}$), and observe the {\it correct} mean-field term say $\mu_t$, at every time step $t$. The agents then use
\begin{align}\label{closed_loop}
\hat{g}({\mu}_t) = \hat{\gamma}(\cdot|x_t,{\mu}_t)
\end{align}
to select their control actions $u_t$ at time $t$.

We denote the accumulated cost under this policy $\hat{g}$ by $K_\beta(\mu_0,\hat{g})$, and we will compare this with the optimal cost that can be achieved, which is $K_\beta^*(\mu_0)$ for some initial distribution $\mu_0$.

\begin{theorem}\label{closed_inf}
Consider the closed loop policy $\hat{g}$ in (\ref{closed_loop}) designed for an estimate model $\hat{\mathcal{T}},\hat{c}$ which satisfies (\ref{miss}) for the infinite population dynamics. Under Assumption \ref{main_assmp}, if $\beta K<1$
\begin{align*}
K_\beta(\mu_0,\hat{g}) - K_\beta^*(\mu_0) \leq 2\lambda\frac{(\beta C-\beta K +1)}{(1-\beta)^2(1-\beta K)}
\end{align*}
where $K=(K_f + \delta_T)$ and $C=(\|c\|_\infty + K_c)$.
\end{theorem}
\begin{proof}
We start with the following upper-bound
\begin{align}\label{f_diff}
K_\beta(\mu,\hat{g}) - K_\beta^*(\mu)\leq \left| K_\beta(\mu,\hat{g})  - \hat{K}_\beta(\mu)\right| + \left|\hat{K}_\beta(\mu) -  K_\beta^*(\mu)\right|
\end{align}
where $\hat{K}_\beta(\mu)$ denotes the optimal value function for the mismatched model.  We have an upper-bound for the second term by Lemma \ref{mod_dis}. We write the following Bellman equations for the first term:
\begin{align*}
&K_\beta(\mu,\hat{g}) = k(\mu,\hat{\gamma}) + \beta K_\beta\left( F(\mu,\hat{\gamma}) , \hat{g}  \right)\\
&\hat{K}_\beta(\mu) = \hat{k}(\mu,\hat{\gamma}) + \beta \hat{K}_\beta\left( \hat{F}(\mu,\hat{\gamma})  \right).
\end{align*}
We can then write 
\begin{align*}
\left| K_\beta(\mu,\hat{g})  - \hat{K}_\beta(\mu)\right| &\leq \left|  k(\mu,\hat{\gamma}) - \hat{k}(\mu,\hat{\gamma})\right|\\
& + \beta\left|  K_\beta\left( F(\mu,\hat{\gamma}) , \hat{g}  \right) - \hat{K}_\beta\left( F(\mu,\hat{\gamma})\right)\right|\\
& + \beta \left|  \hat{K}_\beta\left( F(\mu,\hat{\gamma})\right) -  {K}^*_\beta\left( F(\mu,\hat{\gamma})\right)\right|\\
&+ \beta \left|   {K}^*_\beta\left( F(\mu,\hat{\gamma})\right) - K_\beta^*(\hat{F}(\mu,\hat{\gamma}) )\right|\\
& + \beta \left|  K_\beta^*(\hat{F}(\mu,\hat{\gamma}) ) - \hat{K}_\beta\left( \hat{F}(\mu,\hat{\gamma})  \right)\right|
\end{align*}
We note that $ \left|  k(\mu,\hat{\gamma}) - \hat{k}(\mu,\hat{\gamma})\right| \leq \lambda$ and $ \|F(\mu,\hat{\gamma}) - \hat{F}(\mu,\hat{\gamma}) \| \leq \lambda$. Using Lemma \ref{mod_dis} for the third and the last terms above, we get
\begin{align*}
\left| K_\beta(\mu,\hat{g})  - \hat{K}_\beta(\mu)\right| &\leq \lambda + \beta \sup_\mu \left| K_\beta(\mu,\hat{g})  - \hat{K}_\beta(\mu)\right| \\
& +  2\lambda \beta\left( \frac{\beta C- \beta K +1}{(1-\beta)(1-\beta K)}  \right) + \beta \|K_\beta^*\|_{Lip} \lambda.
\end{align*}
Rearranging the terms and taking the supremum on the left hand side over $\mu\in\P(\mathds{X})$,  and noting that $\|K_\beta^*\|_{Lip} \leq \frac{C}{1-\beta K}$ we can then write
\begin{align*}
\left| K_\beta(\mu,\hat{g})  - \hat{K}_\beta(\mu)\right| &\leq \frac{\lambda}{(1-\beta)}\left(1 +   2\beta\left( \frac{\beta C- \beta K +1}{(1-\beta)(1-\beta K) }\right)   + \frac{\beta C}{(1-\beta K)}  \right) \\
& = \lambda \left( \frac{(1+\beta)(\beta C - \beta K +1)}{(1-\beta)^2 (1-\beta K)}  \right)
\end{align*}
Combining this bound, and Lemma \ref{mod_dis} with (\ref{f_diff}), we can conclude the proof.

\end{proof}

\subsubsection{Error Bounds for Open Loop Control}  
We assume that the agents calculate an optimal policy, say $\hat{g}$ for the incorrect model, and estimate the mean-field flow under the incorrect model with the policy $\hat{g}$. That is, at every time step $t$, the agents use
\begin{align}\label{open_loop}
\hat{g}(\hat{\mu}_t) = \hat{\gamma}(\cdot|x_t,\hat{\mu}_t)
\end{align}
to select their control actions $u_t$ at time $t$. Furthermore, $\hat{\mu}_t$ is estimated with 
\begin{align}\label{open_loop_mu}
\hat{\mu}_{t+1}(\cdot)=\hat{\mathcal{T}}(\cdot|x,u,\hat{\mu}_t)\hat{\gamma}(du|x,\hat{\mu}_t)\hat{\mu}_t(dx)
\end{align}
where $\hat{\mathcal{T}}$ is the learned and possibly incorrect model.
We are then interested in the optimality gap given by
\begin{align*}
K_\beta(\mu_0,\hat{g}) - K_\beta^*(\mu_0)
\end{align*}
where $K_\beta(\mu_0,\hat{g})$ denotes the accumulated cost when the agents follow the open loop policy $\hat{g}(\hat{\mu}_t) = \hat{\gamma}(\cdot|x_t,\hat{\mu}_t)$ at every time $t$. We note that the distinction from the closed loop control is that $\hat{\mu}_t$ is not observed but estimated using the model $\hat{\mathcal{T}}$.

\begin{theorem}
Consider the open loop policy $\hat{g}$ in (\ref{closed_loop}) which is designed for an estimate model that satisfies (\ref{miss}) for the infinite population dynamics. Under Assumption \ref{main_assmp}, if $\beta K<1$,
\begin{align*}
K_\beta(\mu_0,\hat{g}) - K_\beta^*(\mu_0) \leq  2\lambda\frac{\beta (C-K) +1}{(1-\beta)(1-\beta K)}
\end{align*}
for any $\mu_0\in\P(\mathds{X})$ where $C=\|c\|_\infty +  K_c$ and $K=K_f + \delta_T$.
\end{theorem}

\begin{proof}
We start with the following upper-bound
\begin{align}\label{init_bound}
K_\beta(\mu_0,\hat{g}) - K_\beta^*(\mu_0)\leq \left| K_\beta(\mu_0,\hat{g})  - \hat{K}_\beta(\mu_0)\right| + \left|\hat{K}_\beta(\mu_0) -  K_\beta^*(\mu_0)\right|
\end{align}
We have an upper-bound for the second term by Lemma \ref{mod_dis}. We now focus on the first term:
\begin{align*}
 \left| K_\beta(\mu_0,\hat{g})  - \hat{K}_\beta(\mu_0)\right| \leq \sum_{t=0}^\infty \beta^t \left| k(\mu'_t,\hat{\gamma}_t) - \hat{k}(\hat{\mu}_t,\hat{\gamma}_t) \right|
\end{align*}
where we write $\hat{\gamma}_t:=\hat{\gamma}(\cdot|x,\hat{\mu}_t)$, and $\mu'_t$ denotes the measure flow under the true dynamics with the incorrect policy $\hat{\gamma}_t$, that is 
\begin{align*}
\mu'_{t+1} = \hat{F}(\mu'_t,\hat{\gamma}_t) := \int {\mathcal{T}}(\cdot|x,u,\mu'_t) \hat{\gamma}(du|x,\hat{\mu}_t) \mu'_t(dx).
\end{align*}
We next claim that
\begin{align*}
\|\mu'_{t} - \hat{\mu}_t\| \leq \lambda \sum_{n=0}^{t-1} (\delta_T + K_f)^n.
\end{align*}
We show this by induction. For $t=1$, we have that
\begin{align*}
\|\mu'_{1} - \hat{\mu}_1\| &= \left\|\int  {\mathcal{T}}(\cdot|x,u,\mu_0) \hat{\gamma}(du|x, {\mu}_0) \mu_0(dx) -  \int \hat{\mathcal{T}}(\cdot|x,u,\mu_0) \hat{\gamma}(du|x,{\mu}_0) \mu_0(dx)\right\|\\
&\leq \lambda.
\end{align*}
We now assume that the claim is true for $t$:
\begin{align*}
\|\mu'_{t+1} - \hat{\mu}_{t+1}\| &= \bigg\|  \int {\mathcal{T}}(\cdot|x,u,\mu'_t) \hat{\gamma}(du|x,\hat{\mu}_t) \mu'_t(dx) -  \int \hat{\mathcal{T}}(\cdot|x,u,\hat{\mu}_t) \hat{\gamma}(du|x,\hat{\mu}_t) \hat{\mu}_t(dx)\bigg\|\\
&\leq \bigg\|  \int {\mathcal{T}}(\cdot|x,u,\mu'_t) \hat{\gamma}(du|x,\hat{\mu}_t) \mu'_t(dx) -  \int {\mathcal{T}}(\cdot|x,u,{\mu}'_t) \hat{\gamma}(du|x,\hat{\mu}_t) \hat{\mu}_t(dx)\bigg\|\\
& \quad + \bigg\|  \int {\mathcal{T}}(\cdot|x,u,\mu'_t) \hat{\gamma}(du|x,\hat{\mu}_t) \hat{\mu}_t(dx) -  \int \hat{\mathcal{T}}(\cdot|x,u,\hat{\mu}_t) \hat{\gamma}(du|x,\hat{\mu}_t) \hat{\mu}_t(dx)\bigg\|\\
&\leq \delta_T \|\mu'_t - \hat{\mu}_t\| +  \sup_{x,u}\left\|{\mathcal{T}}(\cdot|x,u,\mu'_t)  -  \hat{\mathcal{T}}(\cdot|x,u,\hat{\mu}_t)\right\|\\
& \leq  \delta_T \|\mu'_t - \hat{\mu}_t\| +  \sup_{x,u}\left\|{\mathcal{T}}(\cdot|x,u,\mu'_t)  -  {\mathcal{T}}(\cdot|x,u,\hat{\mu}_t)\right\| \\
& \qquad \qquad\qquad \quad +   \sup_{x,u}\left\|{\mathcal{T}}(\cdot|x,u,\hat{\mu}_t)  -  \hat{\mathcal{T}}(\cdot|x,u,\hat{\mu}_t)\right\|   \\
& \leq (\delta_T + K_f) \|\mu'_t - \hat{\mu}_t\| + \lambda \\
& \leq (\delta_T + K_f)  \lambda \sum_{n=0}^{t-1}(\delta_T + K_f)^n  +\lambda  =  \lambda \sum_{n=0}^{t} (\delta_T + K_f)^n.
\end{align*}
where we used the induction argument at the last inequality.  We now go back to:
\begin{align*}
 \left| K_\beta(\mu_0,\hat{g})  - \hat{K}_\beta(\mu_0)\right| \leq \sum_{t=0}^\infty \beta^t \left| k(\mu'_t,\hat{\gamma}_t) - \hat{k}(\hat{\mu}_t,\hat{\gamma}_t) \right|.
\end{align*}
For the term inside the summation, we write
\begin{align*}
\left| k(\mu'_t,\hat{\gamma}_t) - \hat{k}(\hat{\mu}_t,\hat{\gamma}_t) \right|& = \left| \int c(x,u,\mu_t')\hat{\gamma}(du|x,\hat{\mu}_t) \mu'_t(dx) - \int \hat{c}(x,u,\hat{\mu}_t)\hat{\gamma}(du|x,\hat{\mu}_t)\hat{\mu}_t(dx)\right|\\
&\leq  \left| \int c(x,u,\mu_t')\hat{\gamma}(du|x,\hat{\mu}_t) \mu'_t(dx) - \int {c}(x,u,\mu'_t)\hat{\gamma}(du|x,\hat{\mu}_t)\hat{\mu}_t(dx)\right|\\
& \quad +  \left| \int c(x,u,\mu_t')\hat{\gamma}(du|x,\hat{\mu}_t) \hat{\mu}_t(dx) - \int \hat{c}(x,u,\hat{\mu}_t)\hat{\gamma}(du|x,\hat{\mu}_t)\hat{\mu}_t(dx)\right|\\
&\leq \|c\|_\infty \|\mu'_t - \hat{\mu}_t\| + \sup_{x,u}\left| c(x,u,\mu'_t) - \hat{c}(x,u, \hat{\mu}_t) \right|\\
&\leq \|c\|_\infty \|\mu'_t - \hat{\mu}_t\| + \sup_{x,u}\left| c(x,u,\mu'_t) - c(x,u, \hat{\mu}_t) \right|\\
&\qquad\qquad + \sup_{x,u}\left| c(x,u, \hat{\mu}_t) - \hat{c}(x,u, \hat{\mu}_t) \right|\\
&\leq (\|c\|_\infty + K_c)  \|\mu'_t - \hat{\mu}_t\| + \lambda.
\end{align*}
Using this bound, we finalize our argument. In the following we denote by $K:= (K_f + \delta_T)$ and $C:=(\|c\|_\infty + K_c)$ to conclude:
\begin{align*}
& \left| K_\beta(\mu_0,\hat{g})  - \hat{K}_\beta(\mu_0)\right| \leq \sum_{t=0}^\infty \beta^t \left| k(\mu'_t,\hat{\gamma}_t) - \hat{k}(\hat{\mu}_t,\hat{\gamma}_t) \right| \\
&\leq  C \sum_{t=0}^\infty \beta^t    \|\mu'_t - \hat{\mu}_t\| +  \frac{\lambda}{1-\beta}\\
&\leq C \lambda \sum_{t=0}^\infty \beta^t     \sum_{n=0}^{t-1} K^n +  \frac{\lambda}{1-\beta}= C\lambda \sum_{t=0}^\infty \beta^t \frac{1-K^t}{1-K} + \frac{\lambda}{1-\beta}\\
& = \frac{C\lambda}{(1-\beta)(1-K)}  - \frac{C\lambda}{(1-K)(1-\beta K)} + \frac{\lambda}{1-\beta}\\
&= \frac{C\lambda \beta}{(1-\beta)(1-\beta K)} + \frac{\lambda}{1-\beta} = \lambda\frac{\beta (C-K) +1}{(1-\beta)(1-\beta K)}.
\end{align*}
This is the bound for the first term in (\ref{init_bound}), combining this with the upper-bound on the second term in (\ref{init_bound}) by Lemma \ref{mod_dis}, we can complete the proof.
\end{proof}

\begin{lemma}\label{mod_dis}
Under Assumption \ref{main_assmp}, if $\beta K <1$
\begin{align*}
\left| \hat{K}_\beta (\mu_0) - K_\beta^*(\mu_0)\right| \leq \lambda\left( \frac{\beta C- \beta K +1}{(1-\beta)(1-\beta K)}  \right)
\end{align*}
for any initial distribution $\mu_0\in\P(\mathds{X})$ where $C=\|c\|_\infty + K_c$ and $K=K_f + \delta_T$.
\end{lemma}
\begin{proof}
The proof can be found in the appendix \ref{mod_dis_proof}.
\end{proof}

\subsection{Error Bounds for Finitely Many Agents}

We introduce the following constant to denote the expected distance of an empirical measure to its true distribution:
\begin{align}\label{emp_bound}
M_N:= \sup_{\mu\in\P(\mathds{X})} E\left[ \left\|\mu^N - \mu\right\|  \right]\\
\bar{M}_N = \sup_{\mu\in\P(\mathds{X} \times \mathds{U})} E\left[ \left\|\mu^N - \mu\right\|  \right]
\end{align}
where $\mu^N$ is an empirical measure of the distribution $\mu$, and the expectation is with respect to the randomness over the realizations of $\mu^N$.

\begin{remark}
We note that the constants can be bounded in terms of the population size $N$. In particular, for the finite space $\mathds{X}$ and $\mathds{U}$ 
\begin{align*}
M_N\leq  \frac{K}{\sqrt{N}}
\end{align*}
where $K<\infty$ in general depends on the  underlying space $\mathds{X}$ (or the space $\mathds{X}\times\mathds{U}$ for $\bar{M}_N$). Furthermore, for continuous state and action spaces, e.g. for $\mathds{X}\subset \mathds{R}^d$, the empirical error term is in the order of $O(N^{\frac{-1}{2d}})$.
\end{remark}

\subsubsection{Error Bounds for Open Loop Control}
In this section, we will study the case where each agent in an $N$-agent control system follows the open-loop control given by the solution of the infinite population control problem with mismatched model estimation. We summarize this for some agent $i$ as follows:
\begin{itemize}
\item Collectively agree on a policy $\hat{g}$ as in (\ref{open_loop}) according to the agreed upon estimate model $\hat{\mathcal{T}},\hat{c}$ that satisfies (\ref{miss}). Note that this policy is an optimal policy for the infinite population dynamics under the estimate model.
\item Estimate the mean-field term $\hat{\mu}_t$ at time $t$ according to (\ref{open_loop_mu}) using the approximate model $\hat{\mathcal{T}}$
\item Find the randomized agent level policy $\hat{\gamma}$ using $\hat{g}(\hat{\mu}_t) = \hat{\gamma}(\cdot|x_t^i,\hat{\mu}_t)$
\item Observe local state $x_t^i$, and apply action $u_t^t\sim\hat{\gamma}(\cdot|x_t^i,\hat{\mu}_t)$.
\end{itemize}
If every agent follows this policy, we have the following upperbound for the performance loss compared to the optimal value of the $N$-population control problem,
\begin{theorem}
Under Assumption \ref{main_assmp}, if each agent follows the steps summarized above, we then have that
\begin{align*}
K_\beta^N(\mu^N,\hat{\gamma}) - K_\beta^{N,*}(\mu^N)  \leq 2 \lambda\left( \frac{\beta C- \beta K +1}{(1-\beta)(1-\beta K)}  \right)  + M_N \frac{4\beta C}{(1-\beta)(1-\beta K)}.
\end{align*}
where $C=(\|c\|_\infty + K_c)$, and $K=(K_f + \delta_T)$.
\end{theorem}
\begin{proof}
 For some $\hat{\mu}_0=\mu_0 = \mu_{\bf x_0}=\mu^N$
\begin{align*}
K_\beta^N(\mu^N,\hat{\gamma}) - K_\beta^{N,*}(\mu^N)  \leq &\left|K_\beta^N(\mu^N,\hat{\gamma}) - \hat{K}^*_\beta(\mu^N)\right|+ \left|  \hat{K}^*_\beta(\mu^N) - K^*_\beta(\mu^N)\right| \\
&+ \left|  K^*_\beta(\mu^N) - K_\beta^{N,*}(\mu^N) \right|.
\end{align*}
The second term above is bounded by Lemma \ref{mod_dis}, the last term is bounded by Lemma \ref{finite_inf_dif}, finally for the first term we have
\begin{align*}
\left|K_\beta^N(\mu^N,\hat{\gamma}) - \hat{K}^*_\beta(\mu^N)\right| \leq \sum_{t=0}^\infty \beta^t E\left[ \left|k(\mu_{\bf x_t},\hat{\gamma}) - \hat{k}(\hat{\mu}_t,\hat{\gamma})\right|\right]
\end{align*}
For the term inside of the expectation, we have
\begin{align*}
& \left|k(\mu_{\bf x_t},\hat{\gamma}) - \hat{k}(\hat{\mu}_t,\hat{\gamma})\right| \\
&= \left| \int c(x,u,\mu_{\bf x_t})\hat{\gamma}(du|x,\hat{\mu}_t)\mu_{\bf x_t}(dx) - \int \hat{c}(x,u,\hat{\mu}_t)\hat{\gamma}(du|x,\hat{\mu}_t) \hat{\mu}_t(dx)  \right|\\
& \leq \lambda + C\|\mu_{\bf x_t} - \hat{\mu}_t\| 
\end{align*}
where $C=(\|c\|_\infty + K_c)$. We can then write
\begin{align*}
\left|K_\beta^N(\mu^N,\hat{\gamma}) - \hat{K}^*_\beta(\mu^N)\right| &\leq \sum_{t=0}^\infty \beta^t E\left[ \left|k(\mu_{\bf x_t},\hat{\gamma}) - \hat{k}(\hat{\mu}_t,\hat{\gamma})\right|\right] \\
&\leq \sum_{t=0}^\infty \beta^t \left( \lambda + C E\left[\|\mu_{\bf x_t} - \hat{\mu}_t\| \right]\right)\\
& \leq \frac{\lambda}{1-\beta} + C\sum_{t=0}^\infty \beta^t  \sum_{n=0}^{t-1} K^n( \lambda + 2 M_N )\\
&=  \frac{\lambda}{1-\beta}  + \frac{\beta C (\lambda + 2M_N)}{(1-\beta)(1-\beta K)}\\
& = \lambda\left( \frac{\beta C- \beta K +1}{(1-\beta)(1-\beta K)}  \right)  + M_N \frac{2\beta C}{(1-\beta)(1-\beta K)}.
\end{align*}
where we have used Lemma \ref{meas_diff} which is presented below.
\end{proof}

\begin{lemma}\label{meas_diff}
Let $x_t^i$ be the state of the agent $i$ at time $t$ when each agent follows the open-loop policy $\hat{\gamma}(\cdot|x^i_t,\hat{\mu}_t)$ in an $N$-agent control dynamics. We denote by ${\bf x}_t$ the vector of the states of $N$ agents at time $t$. Under Assumption \ref{main_assmp}, we then have that
\begin{align*}
&E\left[ \left\|\mu_{\bf x_{t}} - \hat{\mu}_{t}\right\|\right]  \leq \sum_{n=0}^{t-1} K^n( \lambda + 2 M_N )
\end{align*}
where $K=K_f + \delta_t$, and where the expectation is with respect to the random dynamics of the $N$ player control system.
\end{lemma}

\begin{proof}
The proof can be found in Appendix \ref{meas_diff_proof}.
\end{proof}

\begin{lemma}\label{finite_inf_dif}
Under Assumption \ref{main_assmp}, 
\begin{align*}
\left| K_\beta^{N,*}(\mu^N) - K_\beta^*(\mu^N)\right| \leq \frac{2\beta C}{(1-\beta)(1-\beta K)}M_N
\end{align*}
where $C=(\|c\|_\infty + K_c)$ and $K=(K_f + \delta_T)$, for any $\mu^N \in \P_N(\mathds{X}) \subset \P(\mathds{X})$ that is for any $\mu^N$ that can be achieved with an empirical distribution of $N$ agents.
\end{lemma}
\begin{proof}
The proof can be found in the appendix \ref{finite_inf_dif_proof}.
\end{proof}

\subsubsection{Error Bounds for Closed Loop Control} In this section, we will assume that the agents find and agree on an  optimal policy for the control problem using the agreed-upon mismatched model $\hat{c}, \hat{\mathcal{T}}$ with infinite agent dynamics. However, unlike open-loop control, to execute this policy, they observe the empirical state distribution of the team of $N$-agents, say $\mu_t^N$ at time $t$ and apply $\hat{\gamma}(\cdot|x,\mu_t^N)$. We summarize the application of this policy as follows: 
\begin{itemize}
\item Collectively agree on a policy $\hat{g}$ as in (\ref{closed_loop}) according to the agreed upon estimate model $\hat{\mathcal{T}},\hat{c}$ that satisfies (\ref{miss}). Note that this policy is an optimal policy for the infinite population dynamics under the estimate model.
\item {\it Observe} the {\it correct} mean-field term ${\mu}_t$.
\item Find the randomized agent level policy $\hat{\gamma}$ using $\hat{g}({\mu}_t) = \hat{\gamma}(\cdot|x_t^i,{\mu}_t)$
\item Observe local state $x_t^i$, and apply action $u_t^t\sim\hat{\gamma}(\cdot|x_t^i,{\mu}_t)$.
\end{itemize}

We denote the incurred cost under this policy by $K_\beta^N(\mu^N,\hat{\gamma})$ for some initial state distribution $\mu^N$.
\begin{theorem}
Under Assumption \ref{main_assmp}, if each agent follows the steps summarized above, we then have that
\begin{align*}
K_\beta^N(\mu^N,\hat{\gamma}) - K_\beta^{N,*}(\mu^N)\leq \lambda  \frac{2(\beta C - \beta K +1) }{ (1-\beta)^2(1-\beta K)}  + M_N  \frac{4\beta C}{(1-\beta)(1-\beta K)} 
\end{align*}
where $K=(K_f + \delta_T)$ and $C=(\|c\|_\infty + K_c)$.
\end{theorem}
\begin{proof}
The proof follows very similar steps to the results we have proved earlier. 
 For some $\hat{\mu}_0=\mu_0 = \mu_{\bf x_0}=\mu^N$
\begin{align}\label{closed_init_bound}
K_\beta^N(\mu^N,\hat{\gamma}) - K_\beta^{N,*}(\mu^N)  \leq &\left|K_\beta^N(\mu^N,\hat{\gamma}) - \hat{K}_\beta(\mu^N)\right|+ \left|  \hat{K}_\beta(\mu^N) - K^*_\beta(\mu^N)\right| \nonumber\\
&+ \left|  K^*_\beta(\mu^N) - K_\beta^{N,*}(\mu^N) \right|.
\end{align}
The second term above is bounded by Lemma \ref{mod_dis}, the last term is bounded by Lemma \ref{finite_inf_dif}. For the first term we write the Bellman equations:
\begin{align*}
&K_\beta^N(\mu^N,\hat{\gamma})  = k(\mu^N,\hat{\gamma}) + \beta \int K_\beta^{N}(\mu_1^N,\hat{\gamma})\eta(d\mu_1^N|\mu^N,\hat{\gamma})\\
& \hat{K}_\beta(\mu^N) = \hat{k}(\mu^N,\hat{\gamma}) + \beta \hat{K}_\beta \left(\hat{F}(\mu^N,\hat{\gamma})  \right).
\end{align*}
We can then write
\begin{align*}
\left|K_\beta^N(\mu^N,\hat{\gamma}) - \hat{K}^*_\beta(\mu^N)\right|& \leq \left| k(\mu^N,\hat{\gamma}) - \hat{k}(\mu^N,\hat{\gamma})\right|\\
& +  \beta \int \left| K_\beta^{N}(\mu_1^N,\hat{\gamma}) - \hat{K}_\beta (\mu_1^N)\right| \eta(d\mu_1^N|\mu^N,\hat{\gamma})\\
& + \beta \int \left| \hat{K}_\beta (\mu_1^N) - \hat{K}_\beta  \left(\hat{F}(\mu^N,\hat{\gamma})  \right)\right| \eta(d\mu_1^N|\mu^N,\hat{\gamma}) \\
& \leq \lambda + \sup_\mu \left|K_\beta^N(\mu,\hat{\gamma}) - \hat{K}_\beta(\mu)\right| \\
& + 2 \beta  \lambda\left( \frac{\beta C- \beta K +1}{(1-\beta)(1-\beta K)}  \right)\\
& +  \beta \int \left| {K}^*_\beta (\mu_1^N) - {K}^*_\beta  \left(\hat{F}(\mu^N,\hat{\gamma})  \right)\right| \eta(d\mu_1^N|\mu^N,\hat{\gamma})  
\end{align*}
Using almost identical arguments that we have used in the proof of Lemma \ref{finite_inf_dif} and Lemma \ref{meas_diff}, we can bound the last term as
\begin{align*}
 &\beta \int \left| {K}^*_\beta (\mu_1^N) - {K}^*_\beta  \left(\hat{F}(\mu^N,\hat{\gamma})  \right)\right| \eta(d\mu_1^N|\mu^N,\hat{\gamma})\\
&  \leq  \beta\|K_\beta^*\|_{Lip} \left( M_N +\delta_T M_N + \lambda \right)
\end{align*}
Re arranging the terms and noting that $\|K_\beta^*\|_{Lip} \leq \frac{C}{1-\beta K}$, we can write that 
\begin{align*}
\sup_{\mu\in \P_N(\mathds{X})} \left|K_\beta^N(\mu,\hat{\gamma}) - \hat{K}_\beta(\mu)\right| \leq M_N \frac{2\beta C}{(1-\beta)(1-\beta K)} + \lambda \frac{(1+\beta )(\beta C - \beta K +1) }{ (1-\beta)^2(1-\beta K)}.
\end{align*}
Combining this bound with (\ref{closed_init_bound}), one can show that 
\begin{align*}
K_\beta^N(\mu^N,\hat{\gamma}) - K_\beta^{N,*}(\mu^N) \leq \lambda  \frac{2(\beta C - \beta K +1) }{ (1-\beta)^2(1-\beta K)}  + M_N  \frac{4\beta C}{(1-\beta)(1-\beta K)} 
\end{align*}

\end{proof}

\section{Numerical Study} We now present a numerical example to verify the results we have established  in the earlier sections.

We consider a multi-agent taxi service model where each agent represents a taxi. The state and action spaces are binary such that $\mathds{X}=\mathds{U}=\{0,1\}$. We assume that at any given time a given zone is in either a surge or a non-surge mode. The state variable $X_t^i$ represents the location of the agent $i$
\begin{itemize}
\item $X_t^i=0$ $\rightarrow$ agent is in a surge zone (high demand)
\item $X_t^i=1$ $\rightarrow$ agent is in a non-surge zone (low demand)
\end{itemize}
The action variable represents the movement decisions:
\begin{itemize}
\item $U_t^i=0$ $\rightarrow$ remains where they are
\item $U_t^i=1$ $\rightarrow$ relocates to another area.
\end{itemize}
The cost structure is defined as follows:
\begin{itemize}
\item If an agent is in a non-surge zone ($X_t^i=1$),  they incur a cost 
$S$ due to lost earnings
\item If an agent relocates ($U_t^i=1$), they receive a cost $R$, for movement expenses.
\item Furthermore, to encourage a balanced distribution, we penalize deviations from  $40$\%-$60$\% distribution by introducing a cost $10\times (\mu(0)-0.4)^2$ where $\mu(0)$ is the fraction of agents in the surge zones. 
\end{itemize}

For the dynamics, we assume that a non-surge area has a fixed probability $0.2$ of becoming a surge area in the next time step. Furthermore, we assume that a surge area has a probability $0.7\times\mu(0)+0.2$ of becoming a non-surge area, indicating that as more drivers there are in a surge area ($\mu(0)$ is high), the likelihood of it remaining a surge zone decreases (due to increased supply). This then defines the transition probabilities as follows:
\begin{align*}
&Pr(X_{t+1}^i=1|X_t^i=0,U_t^i=0,\mu)=0.7\times\mu(0)+0.2\\
&Pr(X_{t+1}^i=1|X_t^i=0,U_t^i=1,\mu)=0.8\\
&Pr(X_{t+1}^i=1|X_t^i=1,U_t^i=0,\mu)=0.8\\
&Pr(X_{t+1}^i=1|X_t^i=1,U_t^i=1,\mu)=0.7\times\mu(0)+0.2
\end{align*}
We set the parameters as $R=1$, $S=7$ and $\beta=0.7$. 

\noindent{\bf Near optimality of learned models and infinite population approximations.} 
Figure \ref{val_comp} shows the value loss for different values of the number of agents in the system. We graph the loss functions under 3 settings:
\begin{itemize}
\item The optimal policy for the infinite population model with perfect knowledge of transition dynamics and costs.
\item The estimate policy for the infinite population model, where the transition-cost function is learned using discretization basis functions based on the discretization of the measure space $\P(\mathds{X})$ into 6 subsets (see Section \ref{disc_model}).
\item The estimate policy for the infinite population model, where the transition-cost function is learned using a class of basis functions:
\begin{align*}
{\bf\phi}(\mu)= [1, \mu(0), \mu(0)^2, \mu(0)^3, \sin(\mu(0)), \cos(\mu(0))  ].
\end{align*}
Note that the cost and the transitions are perfectly linear under the basis functions $[1,\mu(0),\mu(0)^2]$. 

\end{itemize}
For the loss, we compare the value of the learned approximate policy with the optimal value in an infinite population environment. Furthermore, we assume that the initial distribution is $\mu_0=1/2\delta_0+1/2\delta_1$. 

In the figure, we also plot a scaled $\frac{1}{\sqrt{N}}$ line which represents the decay rate of the empirical consistency term $M_N$ defined in (\ref{emp_bound}). As verified by the results, the loss in all cases decays at a rate similar to $\frac{1}{\sqrt{N}}$. 

We also observe that the policies for the learned model with polynomial basis functions perform as well as the policies under perfect model knowledge, which is expected as the model is perfectly linear for these basis functions. 

For the learned model under discretization, there is a small performance gap, which is also expected since the model is not perfectly linear under discretization basis functions. Thus, the learned model does not perfectly match the true model under discretization.
\begin{figure}[h]  
    \centering
    \includegraphics[width=0.6\textwidth]{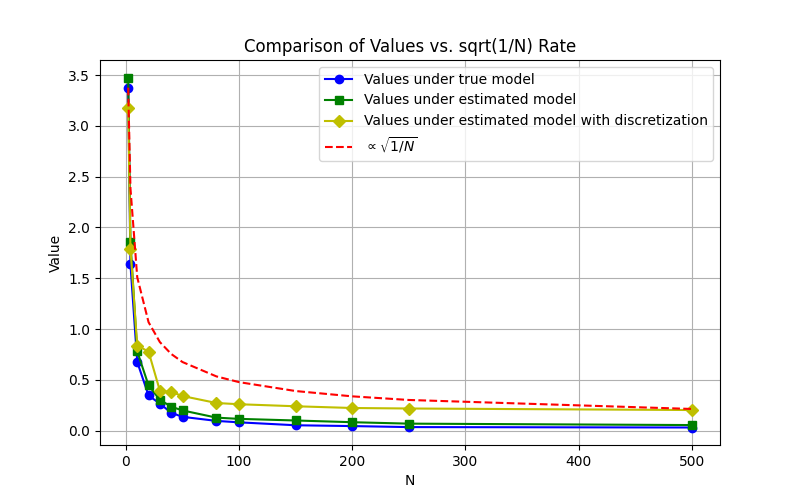} 
    \caption{Value comparison under different policies.}
    \label{val_comp} 
\end{figure}

\noindent {\bf Lack of exploration without common randomness.} Another significant observation from the previous sections about the exploration is also verified in this numerical study. In particular, when agents perform learning individually, we observe that the mean-field term tends to get stuck in certain regions without common randomness. However, if agents choose their actions based on a common randomness, then exploration becomes more efficient as seen in Figure \ref{expo_comp}. In the right graph, the agents follow a policy of the form $\gamma^i(\cdot|x,w^i)$ where $w^i$ is an i.i.d. noise term that is independent across the agents which results in a deterministic flow of the mean-field term, and results in poor exploration. In the graph on the left, the agents follow exploration policies of the form $\gamma(\cdot|x,w^0)$ where $w^0$ is a common noise that is shared by all agents. As a result, the flow of the mean-field term becomes stochastic and a better exploration is observed.
\begin{figure}[h]  
    \centering
    \includegraphics[width=0.45\textwidth]{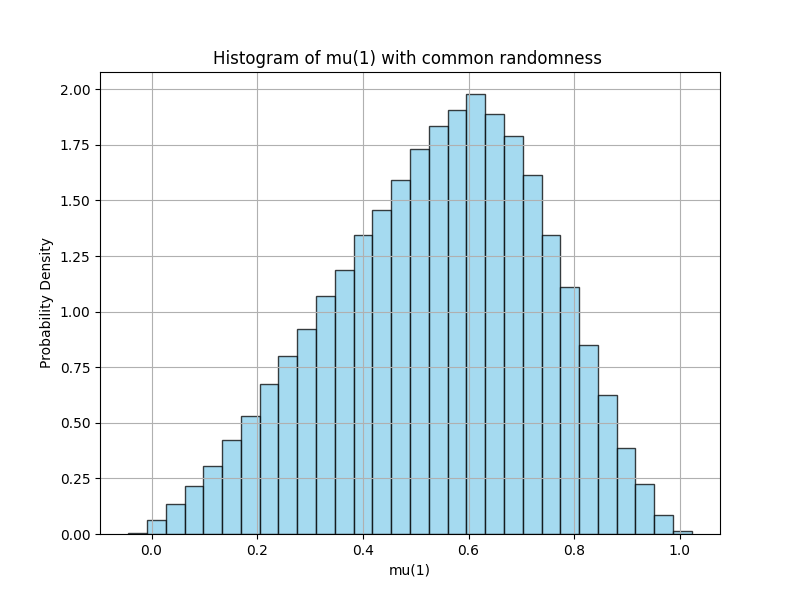} \includegraphics[width=0.45\textwidth]{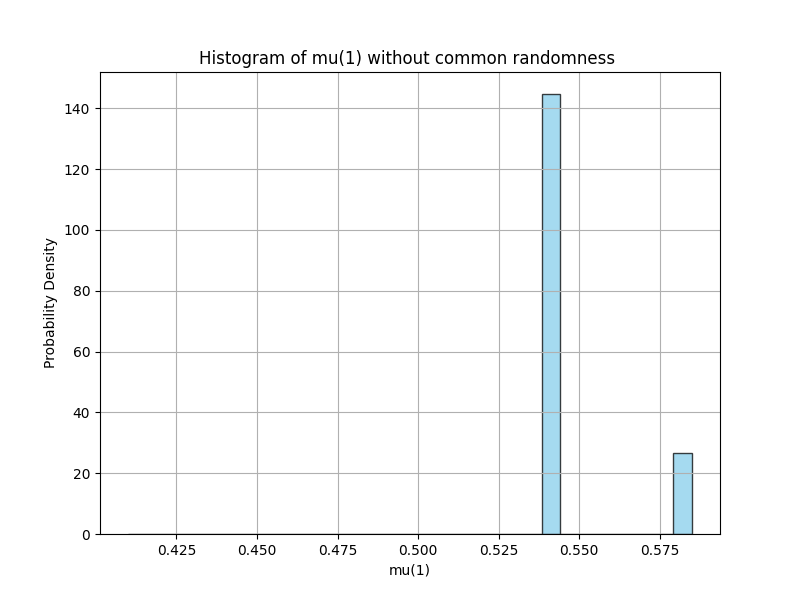} 
    \caption{Learned regions with and without common noise.}
    \label{expo_comp} 
\end{figure}

\section{Conclusion} We have studied model-based learning methods for mean-field control problems using linear function approximations, focusing on both fully coordinated and independent learning approaches. We have observed that full decentralization is generally not possible even when the agents agree on a common model. For the independent learning method, although agents do not need to share their local state information for the convergence, a certain level of coordination is inevitable especially for the exploration phase of the control problem which is done using a common noise process. For the learned models, we have provided error analysis which stems from two main sources (i) modeling mismatch due to linear function approximation, (ii) error arising from the infinite population approximation. 

We have observed that the exploration is a key challenge in the learning of mean-field control systems. The analysis in the paper suggests that the stochastic controllability of the mean-field systems is closely related to the exploration problem. A natural future direction is then to further analysis of the controllability and exploration properties of the mean-field control system.

\acks{E. Bayraktar is partially supported by the National Science Foundation under grant DMS-2106556 and by the Susan M. Smith chair.}

\appendix

\section{Proof of Proposition \ref{sgd}} \label{sgd_proof}

{\bf Step 1}.
We first show that $E\left[\|v_t-v^*\|^2\right]$ remains uniformly bounded over $t$. We write
\begin{align}\label{fb_step}
\|v_{t+1}-v^*\|^2 \leq  \|v_t - v^*\|^2 - 2\alpha_t \langle \nabla g(s_t,v_t), v_t-v^*\rangle + \alpha_t^2 \nabla^2g(s_t,v_t)
\end{align}
For $E[\nabla^2 g(s_t,v_t)] $ we have that 
\begin{align*}
&E[\nabla^2 g(s_t,v_t)]  = E\left[ \left| 2k(s_t)  (k(s_t)^\intercal v_t - h(s_t)  )  \right|^2\right]\leq K E \left[ 2\|k(s_t)\|^2 \|v_t\|^2 + 2\|h(s_t)\|^2\right]\\
& \leq K E\left[   \|v_t\|^2 +1 \right] \leq K \left(E \left[  \|v_t-v^*\|^2 \right] + 1\right)
\end{align*}
where the generic constant $K<\infty$ may represent different values at different steps.  Denoting by $A_t:= E\left[\|v_t-v^*\|^2\right]$, if we take the expectation on both sides of (\ref{fb_step}) we can write
\begin{align}\label{square_step}
A_{t+1}&  \leq A_t - 2\alpha_t E\left[ \langle \nabla g(s_t,v_t), v_t - v^* \rangle\right] + \alpha_t^2 K A_t +  \alpha_t^2 K     \nonumber\\
&\leq  A_t - 2\alpha_t E\left[ g(s_t,v_t) - g(s_t,v^*) \right] + \alpha_t^2 K A_t +  \alpha_t^2 K 
\end{align}
where at the last step we used the convexity of $g(s_t,v_t)$ for every $s_t$. 
 We now introduce $\hat{s}_t$ which are independent over $t$ and each $\hat{s}_t$ is distributed according to $\pi(\cdot)$.  For the middle term above we write
\begin{align}\label{middle_term}
 -2\alpha_t E\left[ g(s_t,v_t) - g(s_t,v^*) \right] =  & -2\alpha_t E\left[   \left(g(s_t,v_t) - g(\hat{s}_t,v_t)\right)\right]  \nonumber\\
& -2\alpha_t   E\left[   \left(g(\hat{s}_t,v_t) - g(\hat{s}_t,v^*)\right)\right] \nonumber \\
& -2\alpha_t   E\left[  \left(g(\hat{s}_t,v^*) - g(s_t,v^*)\right)\right] 
\end{align}
where the expectation is with respect to the independent coupling between $s_t,\hat{s}_t$.

 We denote by 
\begin{align*}
&b^1_t =  -2\alpha_t E\left[   \left(g(s_t,v_t) - g(\hat{s}_t,v_t)\right)\right] \\
&b_t^2 = -2\alpha_t   E\left[   \left(g(\hat{s}_t,v_t) - g(\hat{s}_t,v^*)\right)\right] \nonumber \\
&b_t^3 = -2\alpha_t   E\left[  \left(g(\hat{s}_t,v^*) - g(s_t,v^*)\right)\right] .
\end{align*}
For $b_t^1$, we consider its absolute value to and write
\begin{align}\label{b_1}
&|b_t^1| \leq 2\alpha_t E\left[   \left|g(s_t,v_t) - g(\hat{s}_t,v_t)\right|\right] \nonumber\\
&\leq 2\alpha_t E\left[  \left| (k(s_t)^\intercal v_t  -h(s_t) )^2 -    (k(\hat{s}_t)^\intercal v_t  -h(\hat{s}_t) )^2  \right|\right] \nonumber\\
&\leq  2\alpha_t E\left[    \left| \left(k(s_t)^\intercal - k(\hat{s}_t)^\intercal \right) v_t\right| \left| \left(k(s_t)^\intercal + k(\hat{s}_t)^\intercal \right)v_t - h(s_t) - h(\hat{s}_t)  \right|    \right]\nonumber\\
& \leq 2\alpha_tE\left[  \left( 2\|k\|_\infty \| k(s_t) - k(\hat{s}_t)\| \|v_t\|^2 +   2\|h\|_\infty \| k(s_t) - k(\hat{s}_t)\| \|v_t\|  \right)  \right]\nonumber\\
& \leq  2\alpha_t  K E\left[  \|k(s_t) - k(\hat{s}_t)\| \right] E \left[  \|v_t\|^2 \right] +   2\alpha_t  K E\left[ K \|k(s_t) - k(\hat{s}_t)\| \right] E \left[  \|v_t\| \right]  \nonumber\\
&\leq 2\alpha_t K E\left[  \|k(s_t) - k(\hat{s}_t)\| \right] E \left[  \|v_t -v^*\|^2 \right] 
\end{align}
where we used generic constant $K<\infty$ for the above analysis that might have different values at different steps. Furthermore, we used the inequality $\|v_t\|^2 \leq 2 \|v_t-v^*\|^2 + 2\|v^*\|^2$.   We also assume that $\|v_t\| \geq 1$ to use $\|v_t\| \leq \|v_t\|^2$, note that this is without loss of generality as we are trying to show that $E\|v_t-v^*\|^2$ is bounded, and for $\|v_t\| \leq 1$, the boundedness is immediate.   For the following analysis, we will denote by
\begin{align*}
\epsilon_t:= E\left[  \|k(s_t) - k(\hat{s}_t)\| \right].
\end{align*}
We now consider the series $\sum_{t=1}^\infty \alpha_t \epsilon_t$. Since $s_t$ is ergodic with a geometric rate with invariant measure $\pi(\cdot)$ and $\hat{s}_t\sim \pi(\cdot)$, we have that
\begin{align}\label{erg_ces_sum}
\sum_{t=1}^\infty \alpha_t\epsilon_t <\infty
\end{align}

We now go back to (\ref{square_step}) 
\begin{align*}
A_{t+1}&  \leq  A_t - 2\alpha_t E\left[ g(s_t,v_t) - g(s_t,v^*) \right] + \alpha_t^2 K A_t +  \alpha_t^2 K  \\
& \leq A_t + |b_t^1| + b_t^2 + b_t^3 +\alpha_t^2 K A_t\\
& \leq A_t + 2\alpha_t K \epsilon_t A_t + 2\alpha_t K \epsilon_t + b_t^2 + b_t^3 + \alpha_t^2 K A_t +  \alpha_t^2 K  \\
&\leq (1 + 2\alpha_t K \epsilon_t +   \alpha_t^2 K )A_t + b_t^2 + b_t^3  +  \alpha_t^2 K .
\end{align*}
For the following we denote by 
\begin{align*}
c_t =(1 + 2\alpha_t K \epsilon_t +   \alpha_t^2 K )
\end{align*}
Note that one can show the infinite product $\prod_{t=1}^\infty c_t$ converges if and only if the sum
\begin{align*}
\sum_{t=1}^\infty2\alpha_t K \epsilon_t +   \alpha_t^2 K 
\end{align*}
converges. We have shown that the sum  $\sum_{t=1}^\infty \alpha_t\epsilon_t$ is convergent due to geometric ergodicity, and we also have that $\alpha_t^2$ is summable. Thus, we write
\begin{align*}
\prod_{t=1}^\infty c_t < C
\end{align*}
for some $C<\infty$. One can then iteratively show that 
\begin{align*}
A_{t+1}& \leq \prod_{n=1}^{t} c_n A_0 + C \sum_{n=1}^t \left(b_n^2 + b_n^3   +  \alpha_n^2 K  \right)\\
& \leq C A_0 + C \sum_{n=1}^t \left(b_n^2 + b_n^3 + \alpha_n^2 K    \right).
\end{align*}
Consider $b^2_n=  - 2\alpha_n E\left[ g(\hat{s}_n,v_n) - g(\hat{s}_n,v^*) \right] $; since $\hat{s}_t\sim \pi(\cdot)$ for all $t$, and since $v^*=\argmin_v G(v) = \argmin_v \int g(s,v)\pi(ds)$, $b^2_n\leq 0$ for all $n$.    Thus, we can simply remove $b_n^2$ terms to get a further upperbound. For $b_n^3$, we have that 
\begin{align*}
\sum_{n=1}^\infty |b_n^3| \leq \sum_{t=1}^\infty 2\alpha_t \left|g(\hat{s}_t,v^*) - g(s_t,v^*)\right| <\infty 
\end{align*}
using an identical argument we used to show $\sum \alpha_t \epsilon_t < \infty$. 
       In particular,
\begin{align*}
\lim_{t\to \infty} A_t \leq \lim_{t\to \infty} C A_0 + C \sum_{n=1}^t   \left( b_n^3 + \alpha_n^2 K    \right) < \infty
\end{align*}
which shows that $E\|v_t-v^*\|^2$ is bounded uniformly over $t$, which also implies that $E\|v_t\|^2$ is bounded.

{\bf Step 2}. Now we have the boundedness, we go back to (\ref{square_step}); using the bound on $A_t$ (only for the second $A_t$ in (\ref{square_step})), and summing over the terms, we can write
\begin{align*}
A_{N+1}-A_0 \leq \sum_{t=1}^N \left(A_{t+1} - A_t\right)  \leq \sum_{t=1}^N  -2\alpha_t  E\left[ g(s_t,v_t) - g(s_t,v^*) \right]   + \sum_{t=1}^N  \alpha_t^2K
\end{align*}
again using the boundedness of $A_t$, and the fact that $\sum_{t=1}^\infty\alpha_t^2 <\infty$  and sending $N\to\infty$, we get
\begin{align*}
  E\left[ \sum_{t=1}^\infty 2\alpha_t  \left(g(s_t,v_t) - g(s_t,v^*)\right)\right]  <\infty
\end{align*}
We now introduce $\hat{s}_t$ which are independent over $t$ and each $\hat{s}_t$ is distributed according to $\pi(\cdot)$. We then write
\begin{align}\label{fin_sum}
  &E\left[ \sum_{t=1}^\infty \underbrace{2\alpha_t   \left(g(s_t,v_t) - g(\hat{s}_t,v_t)\right)}_{b_t^1}\right]  \nonumber\\
&+   E\left[ \sum_{t=1}^\infty   \underbrace{ 2\alpha_t \left(g(\hat{s}_t,v_t) - g(\hat{s}_t,v^*)\right)}_{b_t^2}\right] \nonumber \\
& +   E\left[ \sum_{t=1}^\infty  \underbrace{ 2\alpha_t  \left(g(\hat{s}_t,v^*) - g(s_t,v^*)\right)}_{b_t^3}\right]  <\infty
\end{align}
where we overwrite the definitions of $b_t^1$, $b_t^2$, and $b_t^3$ (only changing the signs of these terms, see (\ref{middle_term})). 

Recall the analysis for $b_t^1$ in (\ref{b_1}), together with the uniform boundedness of $A_t= E\left \|v_t - v^*\|^2  \right]$ over $t$, we can write that 
\begin{align*}
  &E\left[ \sum_{t=1}^\infty \left|b_t^1\right|\right]  \leq   \sum_{t=1}^\infty 2\alpha_t  K  E\left[\|k(s_t) - k(\hat{s}_t)\|  \right]  <\infty
\end{align*}
where we exchange the sum and expectation with monotone convergence theorem, and where the last step follows from what we have shown in (\ref{erg_ces_sum}).

For the last term similarly, we have that $E\left[ \sum_{t=1}^\infty b_t^3\right] <\infty$, from (\ref{erg_ces_sum}), since $s_t$ is geometrically ergodic with invariant measure $\pi$ and $\hat{s}_t\sim \pi(\cdot)$ and $v^*$ is fixed. 

Going back to (\ref{fin_sum}), now that we have shown the last and the first terms are finite, we can write
\begin{align*}
&   E\left[ \sum_{t=1}^\infty 2\alpha_t  \left(g(\hat{s}_t,v_t) - g(\hat{s}_t,v^*)\right)\right] <\infty.
\end{align*}
Since $\hat{s}_t$ is i.i.d and distributed according to $\pi(\cdot)$, the above also implies that
\begin{align*}
&   E\left[ \sum_{t=1}^\infty 2\alpha_t  \left(G(v_t) - G(v^*)\right)\right] <\infty
\end{align*}
which in turn implies that 
\begin{align*}
\sum_{t=1}^\infty 2\alpha_t  \left(G(v_t) - G(v^*)\right) <\infty
\end{align*}
almost surely. Furthermore, since $\alpha_t$ is not summable, and $\left(G(v_t) - G(v^*)\right) \geq 0$ (as $v^*$ achieves the minimum of $G(v)$), we must have that 
\begin{align*}
G(v_t)\to G(v^*), \text{ almost surely}.
\end{align*}

\section{Proof of Lemma \ref{mod_dis}}\label{mod_dis_proof}
We begin the proof by writing the Bellman equations
\begin{align*}
&\hat{K}_\beta(\mu) = \hat{k}(\mu,\hat{\gamma}) + \beta \hat{K}_\beta(\hat{F}(\mu,\hat{\gamma}))\\
& K_\beta^*(\mu) = k(\mu,\gamma) +\beta K_\beta^*(F(\mu,\gamma))
\end{align*}
where $\hat{\gamma}$ and $\gamma$ are optimal agent-level policies that achieves the minimum at the right hand side of the Bellman equations respectively. We can use then use the same agent level policies by exchanging them to get the following upper-bound 
\begin{align}\label{mid_dif}
&\left| \hat{K}_\beta (\mu) - K_\beta^*(\mu)\right| \leq |\hat{k}(\mu,\gamma) - k(\mu,\gamma)| + \beta \left| \hat{K}_\beta(\hat{F}(\mu,\gamma)) - K_\beta^*(F(\mu,\gamma))\right|\nonumber\\
&\leq  |\hat{k}(\mu,\gamma) - k(\mu,\gamma)| + \beta \left| \hat{K}_\beta(\hat{F}(\mu,\gamma)) - K_\beta^*(\hat{F}(\mu,\gamma))\right|+ \beta  \left| {K}^*_\beta(\hat{F}(\mu,\gamma)) - K_\beta^*({F}(\mu,\gamma))\right|\nonumber\\
&\leq \lambda + \beta \sup_\mu\left| \hat{K}_\beta (\mu) - K_\beta^*(\mu)\right|  + \beta \|K_\beta^*\|_{Lip} \| \hat{F}(\mu,\gamma) - F(\mu,\gamma)\|
\end{align}
We have that 
\begin{align*}
& \hat{F}(\mu,\gamma) - F(\mu,\gamma)\| \leq \left\| \int \hat{\mathcal{T}}(\cdot|x,u,\mu)\gamma(du|x)\mu(dx) -  \int \mathcal{T}(\cdot|x,u,\mu)\gamma(du|x)\mu(dx)  \right\|\\
&\leq \lambda.
\end{align*}
Hence, by rearranging the terms in (\ref{mid_dif}), we can write
\begin{align*}
&\left| \hat{K}_\beta (\mu) - K_\beta^*(\mu)\right| \leq \frac{\lambda}{1-\beta} \left( 1+\beta \|K_\beta\|_{Lip} \right).
\end{align*}
Finally, a slight modification of \cite[Lemma 6]{bayraktar2023finite} for finite $\mathds{X},\mathds{U}$  can be used to show that
\begin{align*}
\|K_\beta^*\|_{Lip} \leq \frac{C}{1-\beta K}
\end{align*} 
which completes the proof that
\begin{align*}
\left| \hat{K}_\beta (\mu) - K_\beta^*(\mu)\right| \leq \lambda\left( \frac{\beta C- \beta K +1}{(1-\beta)(1-\beta K)}  \right).
\end{align*}

\section{Proof of Lemma \ref{meas_diff}}\label{meas_diff_proof}
 We use the notation $\mu_t= \mu_{\bf x_t}$ for the following analysis. Note that with stochastic realization results, there exists a random variable $v_t$ uniformly distributed on $[0,1]$, and a measurable function $\hat{\gamma}$ such that
\begin{align*}
\hat{\gamma}(x,v_t)
\end{align*}
has the same distribution as  $\hat{\gamma}(\cdot|x,\hat{\mu}_t)$, where we overwrite the notation for simplicity. Let ${\bf \hat{x}}_t$ denote a vector of size $N$ state variables that are distributed according to $\hat{\mu}_t$, i.e. ${\bf \hat{x}}_t= [\hat{x}^1_t,\dots,\hat{x}^N_t]$ such that $\hat{x}_t^i\sim \hat{\mu}_t$ for all $i\in\{1,\dots,N\}$. Furthermore, let ${\bf v}_t$ denote a vector of size $N$ where each element is independent and distributed according to the law of $v_t$. We then study the following conditional expected difference:
\begin{align*}
E\left[ \left\|\mu_{\bf x_{t+1}} - \hat{\mu}_{t+1}\right\|\right] = E \left[  E\left[\left\|\mu_{\bf x_{t+1}} - \hat{\mu}_{t+1}\right\|  |  {\bf x}_t,{\bf \hat{x}_t} , {\bf v_t}  \right]   \right].
\end{align*}

Let $\bf w_t$ denote the vector of size $N$ for the noise variables of the agents at time $t$.  Note that we have ${\bf x}_{t+1} = f({\bf x}_t,{\bf u}_t,{\bf w}_t)$ where $u^i_t = \hat{\gamma}(x^i_t,v^i_t)$ for each $i$.   We also introduce  $\hat{\bf u}_t$  such that $\hat{u}^i_t = \hat{\gamma}(\hat{x}^i_t,v_t^i)$.     

We further introduce another vector of noise variables $\hat{\bf w}_t$ where each element is independently distributed, and the distribution of $\hat{w}_t$ agrees the with the kernel $\hat{\mathcal{T}}(\cdot|x,u,\hat{\mu}_t)$. In other words, we use the functional representation of $\hat{\mathcal{T}}(\cdot|x,u,\hat{\mu}_t)$ where 
\begin{align*}
 \hat{f} (x,u,\hat{\mu}_t,\hat{w}_t) \sim \hat{\mathcal{T}}(\cdot|x,u,\hat{\mu}_t)
\end{align*}
for some measurable $\hat{f}$.

We denote by ${\bf P}(d{\bf w}_t) = P(dw^1_t) \times \dots \times P(dw_t^N)$ denote the distribution of the vector ${\bf w}_t$ where it is assumed that $w_t^i$ and $w_t^j$ are independent for all $i\neq j$. ${\bf \hat{w}}_t$ is also distributed according to ${\bf P}(\cdot)$. For the joint distribution of ${\bf w_t,\hat{w}_t}$, we use a coupling of the form
\begin{align*}
{\bf \Omega}(d {\bf w}_t,d{\bf \hat{w}}_t) = \Omega^1(dw_t^1 ,d\hat{w}_t^1)\times \dots \times \Omega^N(dw_t^N,d\hat{w}_t^N).
 \end{align*}
That is, we assume independence over $i\in{1,\dots,N}$, however, an arbitrary coupling is assumed between the distribution of $w_t^i,\hat{w}_t^i$.  We will later specify the particular selection of coordinate wise couplings  $\Omega^1,\dots,\Omega^N$, however, the following analysis will hold correct for a general selection of $\Omega^1,\dots,\Omega^N$.

 For given realizations of $ {\bf x}_t,{\bf \hat{x}_t} , {\bf v_t}$, we write
\begin{align}\label{tri_bound}
&E\left[\left\|\mu_{\bf x_{t+1}} - \hat{\mu}_{t+1}\right\|  |  {\bf x}_t,{\bf \hat{x}_t} , {\bf v_t}  \right] = \int \left\|\mu_{ f({\bf x}_t,{\bf u}_t,{\bf w}_t)} - \hat{\mu}_{t+1}\right\| P(d{\bf w}_t)\nonumber\\
&= \int  \left\|\mu_{ f({\bf x}_t,{\bf u}_t,{\bf w}_t)} - \hat{\mu}_{t+1}\right\| {\bf \Omega}(d{\bf w}_t, d\hat{\bf w}_t)  \nonumber \\
&\leq  \int  \left\|\mu_{ f({\bf x}_t,{\bf u}_t,{\bf w}_t)} - {\mu}_{\hat{f}(\hat{\bf x}_t, \hat{\bf u}_t , \hat{\bf w}_t) }\right\| {\bf \Omega}(d{\bf w}_t, d\hat{\bf w}_t) +  \int  \left\| {\mu}_{\hat{f}(\hat{\bf x}_t, \hat{\bf u}_t , \hat{\bf w}_t) } -  \hat{\mu}_{t+1}\right\| {\bf \Omega}(d{\bf w}_t, d\hat{\bf w}_t)
\end{align} 
Note that $\hat{\bf x}_t$ is a vector of size $N$ where each entry is independent and distributed according to $\hat{\mu}_t$. Furthermore, $\hat{u}^i_t = \hat{\gamma}(\hat{x}^i_t,v^i_t)$, and thus $\hat{u}^i_t \sim \hat{\gamma}(\cdot|\hat{x}^i_t, \hat{\mu}_t)$ for each $i\in\{1,\dots,N\}$. Thus,  $ {\mu}_{\hat{f}(\hat{\bf x}_t, \hat{\bf u}_t , \hat{\bf w}_t) } $ is an empirical measure for $\hat{\mu}_{t+1}$
For the second term above, we then have:
\begin{align}\label{second_term}
  &E \left[\int  \left\| {\mu}_{\hat{f}(\hat{\bf x}_t, \hat{\bf u}_t , \hat{\bf w}_t) } -  \hat{\mu}_{t+1}\right\| {\bf \Omega}(d{\bf w}_t, d\hat{\bf w}_t)\right] =  E\left[ \int  \left\| {\mu}_{\hat{f}(\hat{\bf x}_t, \hat{\bf u}_t , \hat{\bf w}_t) } -  \hat{\mu}_{t+1}\right\| P(d\hat{\bf w}_t)\right]\nonumber\\
&= E\left[E\left[ \left\| \mu_{\hat{\bf x}_{t+1}} - \hat{\mu}_{t+1} \right \| | \hat{\bf x}_t , \hat{\bf u}_t   \right]\right] = E\left[ \| \mu_{\hat{\bf x}_{t+1}} - \hat{\mu}_{t+1}\| \right] \leq M_N
\end{align}
see $(\ref{emp_bound})$ for the definition of $M_N$.

For the first term in (\ref{tri_bound}); we note that $\mu_{ f({\bf x}_t,{\bf u}_t,{\bf w}_t)} $ and $ {\mu}_{f(\hat{\bf x}_t, \hat{\bf u}_t , \hat{\bf w}_t) }$ are empirical measures, and thus 
for every given realization of ${\bf w}_t$ and ${\bf \hat{w}}_t$, the Wasserstein distance is achieved with a particular permutation of ${ f({\bf x}_t,{\bf u}_t,{\bf w}_t)}$ and ${\hat{f}(\hat{\bf x}_t, \hat{\bf u}_t , \hat{\bf w}_t) }$ combined together. That is, letting $\sigma$ denote a permutation map for the vector ${\hat{f}(\hat{\bf x}_t, \hat{\bf u}_t , \hat{\bf w}_t) }$. we have
\begin{align*}
 \left\|\mu_{ f({\bf x}_t,{\bf u}_t,{\bf w}_t)} - {\mu}_{\hat{f}(\hat{\bf x}_t, \hat{\bf u}_t , \hat{\bf w}_t) }\right\|  = \inf_\sigma \frac{1}{N}\sum_{i=1}^N |f(x_t^i,u_t^i,w_t^i,\mu_{\bf x_t}) - \sigma(\hat{f}(\hat{x}_t^i,\hat{u}_t^i,\hat{w}_t^i,\hat{\mu}_t))|.
\end{align*}
 We will however, consider a particular permutation where 
\begin{align}\label{perm_ach}
&\left\| \int  \mathcal{T}(\cdot|x,u,\mu_{{\bf x}_t}) \mu_{\bf (x_t,u_t)}(du,dx) - \hat{\mathcal{T}}(\cdot|x,u,\hat{\mu}_t) \mu_{(\bf \hat{x}_t , \hat{u}_t)}(du,dx)\right\|\nonumber\\
&= \frac{1}{N}\sum_{i=1}^N \left\| \mathcal{T}(\cdot|x_t^i,u_t^i,\mu_{{\bf x}_t}) -\sigma(\hat{\mathcal{T}}(\cdot|\hat{x}_t^i,\hat{u}_t^u,\hat{\mu}_t))\right\|
\end{align} 
For the following analysis, we will drop the permutation notation $\sigma$ and assume that the given order of $\hat{f}(\hat{\bf x}_t, \hat{\bf u}_t , \hat{\bf w}_t) $ achieves the Wasserstein distance in (\ref{perm_ach}). Furthermore, the coupling ${\bf \Omega}$ is assumed to have the same order of coordinate-wise coupling.

 We then write
\begin{align}\label{bound1}
 & \int  \left\|\mu_{ f({\bf x}_t,{\bf u}_t,{\bf w}_t)} - {\mu}_{\hat{f}(\hat{\bf x}_t, \hat{\bf u}_t , \hat{\bf w}_t) }\right\| {\bf \Omega}(d{\bf w}_t, d\hat{\bf w}_t) \nonumber\\
&\leq \int \frac{1}{N}\sum_{i=1}^N \left| f(x^i_t,u^i_t,w_t^i, \mu_{\bf x_t}) - \hat{f}(\hat{x}_t^i,\hat{u}_t^i,\hat{w}_t^i,\hat{\mu}_t) \right|  {\bf \Omega}(d{\bf w}_t, d\hat{\bf w}_t)\nonumber \\
& =  \frac{1}{N}\sum_{i=1}^N \int  \left| f(x^i_t,u^i_t,w_t^i, \mu_{\bf x_t}) - \hat{f}(\hat{x}_t^i,\hat{u}_t^i,\hat{w}_t^i,\hat{\mu}_t) \right|  {\bf \Omega}(d{\bf w}_t, d\hat{\bf w}_t)\nonumber \\
&= \frac{1}{N}\sum_{i=1}^N \int  \left| f(x^i_t,u^i_t,w_t^i, \mu_{\bf x_t}) - \hat{f}(\hat{x}_t^i,\hat{u}_t^i,\hat{w}_t^i,\hat{\mu}_t) \right|  { \Omega^i}(dw^i_t, d\hat{w}^i_t). 
\end{align}

The analysis thus far, works for any coupling  ${\bf \Omega}(d{\bf w}_t, d\hat{\bf w}_t)$. In particular, the analysis holds for the coupling that satisfies
\begin{align*}
\| \mathcal{T}(\cdot|x_t,u_t,\mu_{\bf x_t}) - \hat{\mathcal{T}}(\cdot|\hat{x}^i_t,\hat{u}^i_t,\hat{\mu}_t)\| =  \int  \left| f(x^i_t,u^i_t,w_t^i, \mu_{\bf x_t}) - \hat{f}(\hat{x}_t^i,\hat{u}_t^i,\hat{w}_t^i,\hat{\mu}_t) \right|  { \Omega^i}(dw^i_t, d\hat{w}^i_t).
\end{align*}
for every $i$ for some coordinate-wise coupling $\Omega^i(dw_t^i,d\hat{w}_t^i)$. Continuing from the term (\ref{bound1}), we can then write 
\begin{align*}
 & \int  \left\|\mu_{ f({\bf x}_t,{\bf u}_t,{\bf w}_t)} - {\mu}_{\hat{f}(\hat{\bf x}_t, \hat{\bf u}_t , \hat{\bf w}_t) }\right\| {\bf \Omega}(d{\bf w}_t, d\hat{\bf w}_t) \nonumber\\
&\leq \frac{1}{N}\sum_{i=1}^N \int  \left| f(x^i_t,u^i_t,w_t^i, \mu_{\bf x_t}) - \hat{f}(\hat{x}_t^i,\hat{u}_t^i,\hat{w}_t^i,\hat{\mu}_t) \right|  { \Omega^i}(dw^i_t, d\hat{w}^i_t)\\
& = \int \frac{1}{N}\sum_{i=1}^N \left\|  \mathcal{T}(\cdot|x_t,u_t,\mu_{\bf x_t}) - \hat{\mathcal{T}}(\cdot|\hat{x}^i_t,\hat{u}^i_t,\hat{\mu}_t)\right\|\\
&=\left\| \int  \mathcal{T}(\cdot|x,u,\mu_{{\bf x}_t}) \mu_{\bf (x_t,u_t)}(du,dx) - \int\hat{\mathcal{T}}(\cdot|x,u,\hat{\mu}_t) \mu_{(\bf \hat{x}_t , \hat{u}_t)}(du,dx)\right\|
\end{align*}
where the last step follows from the particular permutation we consider (see (\ref{perm_ach})).

Furthermore, we also have that
\begin{align}\label{first_term}
&\left\| \int  \mathcal{T}(\cdot|x,u,\mu_{{\bf x}_t}) \mu_{\bf (x_t,u_t)}(du,dx) - \int\hat{\mathcal{T}}(\cdot|x,u,\hat{\mu}_t) \mu_{(\bf \hat{x}_t , \hat{u}_t)}(du,dx)\right\|\nonumber\\
&\leq \left\| \int  \mathcal{T}(\cdot|x,u,\mu_{{\bf x}_t}) \mu_{\bf (x_t,u_t)}(du,dx) - \int{\mathcal{T}}(\cdot|x,u,\mu_{{\bf x}_t}) \mu_{(\bf \hat{x}_t , \hat{u}_t)}(du,dx)\right\|\nonumber\\
& + \left\|    \int{\mathcal{T}}(\cdot|x,u,\mu_{{\bf x}_t}) \mu_{(\bf \hat{x}_t , \hat{u}_t)}(du,dx) -  \int{\mathcal{T}}(\cdot|x,u, \hat{\mu}_t) \mu_{(\bf \hat{x}_t , \hat{u}_t)}(du,dx)\right\|\nonumber\\
& + \left\|  \int{\mathcal{T}}(\cdot|x,u, \hat{\mu}_t) \mu_{(\bf \hat{x}_t , \hat{u}_t)}(du,dx) -  \int\hat{\mathcal{T}}(\cdot|x,u,\hat{\mu}_t) \mu_{(\bf \hat{x}_t , \hat{u}_t)}(du,dx)\right\|\nonumber\\
&\leq \delta_T \|\mu_{\bf x_t} - \mu_{\bf \hat{x}_t}\| + K_f \| \mu_{\bf x_t} - \hat{\mu}_t\| + \lambda
\end{align}
where for the first term we use the following bound:
\begin{align*}
&\left\| \int  \mathcal{T}(\cdot|x,u,\mu_{{\bf x}_t}) \mu_{\bf (x_t,u_t)}(du,dx) - \int{\mathcal{T}}(\cdot|x,u,\mu_{{\bf x}_t}) \mu_{(\bf \hat{x}_t , \hat{u}_t)}(du,dx)\right\|\\
& = \left\| \int  \mathcal{T}(\cdot|x,\hat{\gamma}(x,v^i),\mu_{{\bf x}_t})  \mu_{\bf x_t}(dx) - \int{\mathcal{T}}(\cdot|x,\hat{\gamma}(x,v^i),\mu_{{\bf x}_t}) \mu_{\bf \hat{x}_t}(dx)\right\|\\
& \leq \delta_T \|\mu_{\bf x_t} - \mu_{\bf \hat{x}_t}\| 
\end{align*}

Combining (\ref{tri_bound}), (\ref{second_term}), and (\ref{first_term}), we can then write
\begin{align*}
&E\left[ \left\|\mu_{\bf x_{t+1}} - \hat{\mu}_{t+1}\right\|\right]  \leq M_N  + \delta_T E[ \|\mu_{\bf x_t} - \mu_{\bf \hat{x}_t}\| ] + K_f E[ \| \mu_{\bf x_t} - \hat{\mu}_t\| ] + \lambda\\
&\leq (1+\delta_T)M_N + K E[ \| \mu_{\bf x_t} - \hat{\mu}_t\| ]  + \lambda
\end{align*}
where $K=(\delta_T + K_f)$.  Noting that we have assumed $\mu_{\bf x_0} = \hat{\mu}_0$, this bound implies that
\begin{align*}
&E\left[ \left\|\mu_{\bf x_{t}} - \hat{\mu}_{t}\right\|\right]  \leq \sum_{n=0}^{t-1} K^n( \lambda + 2 M_N )
\end{align*}
where we have used the fact that $\delta_T\leq 1$ to simplify the notation.

\section{Proof of Lemma \ref{finite_inf_dif}}\label{finite_inf_dif_proof}
We start by writing the Bellman equations:
\begin{align*}
&K_\beta^*(\mu^N) = k(\mu^N,\gamma_\infty) + \beta K_\beta^*(F(\mu^N,\gamma_\infty))\\
& K_\beta^{N,*}(\mu^N) = k(\mu^N,\Theta^N) + \beta \int K_\beta^{N,*}(\mu_1^N)\eta(d\mu_1^N|\mu^N,\Theta^N)
\end{align*}
where we assume that an optimal selector for the infinite population problem at $\mu^N$ is $\gamma_\infty$ such that the agents should use the randomized agent-level policy $\gamma_\infty(\cdot|x,\mu^N)$. For the $N$-agent problem, we assume that an optimal state-action distribution at $\mu^N$ is given by some $\Theta^N \in \P_N(\mathds{X}\times\mathds{U})$, which can be achieved by some ${\bf x,u,}$ such that $\mu_{\bf x}=\mu^N$ and $\mu_{(\bf x,u)}=\Theta^N$.

We first  assume that $K_\beta^*(\mu^N) > K_\beta^{N,*}(\mu^N)$. For the infinite population problem, instead of using the optimal selector $\gamma_\infty$, we use a randomized agent-level policy from the finite population problem by writing $\Theta^N(du,dx)=\gamma^N(du|x)\mu^N(dx)$, and letting the agents use $\gamma^N$. We emphasize that the optimal state action distribution for  $N$-agents is not achieved if each agent symmetrically use $\gamma^N(du|x)$, in other words, $\gamma^N$ is not an optimal agent-level policy for the $N$-population problem.  To have the equality $\Theta^N(du,dx)=\gamma_N(du|x)\mu^N(dx)$ the number of agents needs to tend to infinity. We can then write
\begin{align*}
& K_\beta^{*}(\mu^N) - K_\beta^{N,*}(\mu^N)\leq  K_\beta(\mu^N,\gamma^N) - K_\beta^{N,*}(\mu^N)\\
&= k(\mu^N,\gamma^N) - k(\mu^N,\Theta^N) + \beta K_\beta^*\left(F(\mu^N,\gamma^N)\right) - \beta \int K_\beta^{N,*}(\mu_1^N)\eta(d\mu_1^N|\mu^N,\Theta^N)
\end{align*}
Note that
\begin{align*}
 k(\mu^N,\gamma^N)  = \int c(x,u,\mu^N)\gamma^N(du|x)\mu^N(dx) = \int c(x,u,\mu^N)\Theta^N(du,dx) = k(\mu^N,\Theta^N).
\end{align*}
 Hence, we can continue:
\begin{align}\label{bound_inf_big}
& K_\beta^{*}(\mu^N) - K_\beta^{N,*}(\mu^N) \leq \beta K_\beta^*\left(F(\mu^N,\gamma^N)\right) - \beta \int K_\beta^{N,*}(\mu_1^N)\eta(d\mu_1^N|\mu^N,\Theta^N)\nonumber\\
& \leq  \beta \int \left| K_\beta^*\left(F(\mu^N,\gamma^N)\right) - K_\beta^*(\mu_1^N)\right| \eta(d\mu_1^N|\mu^N,\Theta^N)\nonumber\\
&\quad + \beta \int \left|  K_\beta^*(\mu_1^N) - K_\beta^{N,*}(\mu_1^N)\right| \eta(d\mu_1^N|\mu^N,\Theta^N)\nonumber\\
& \leq \beta \|K_\beta^*\|_{Lip} \int \left\| F(\mu^N,\gamma^N) - \mu_1^N\right\| \eta(d\mu_1^N|\mu^N,\Theta^N) + \beta \sup_{\mu\in \P_N(\mathds{X})} \left| K_\beta^*(\mu)-  K_\beta^{N,*}(\mu)\right|.
\end{align}
We now focus on the term $ \int \left\| F(\mu^N,\gamma^N) - \mu_1^N\right\| \eta(d\mu_1^N|\mu^N,\Theta^N) $.  We will follow a very similar methodoly as we have used in the proof of Lemma \ref{meas_diff} with slight differences. We denote by ${\bf P}(d{\bf w}) = P(dw^1) \times \dots \times P(dw^N)$ denote the distribution of the vector ${\bf w}$ where it is assumed that $w^i$ and $w^j$ are independent for all $i\neq j$.  Let ${\bf x,u}$ such that $\mu_{\bf x}=\mu^N$ and $\mu_{\bf (x,u)}=\Theta^N$.  We then have that 
\begin{align*}
 \int \left\| F(\mu^N,\gamma^N) - \mu_1^N\right\| \eta(d\mu_1^N|\mu^N,\Theta^N)  = \int \left\|  F(\mu^N,\gamma^N) -\mu_{f(\bf x,u,w)}\right\| {\bf P}(d{\bf w})
\end{align*}
where $f({\bf x,u,w}) = [f(x^1,u^1,w^1,\mu^N),\dots,f(x^N,u^N,w^N,\mu^N ) ]$. We now introduce $(\hat{x}^i,\hat{u}^i) \sim \Theta^N(du,dx)$ where  $i\in\{1,\dots,N\}$, which are different than the state action vectors ${\bf(x,u)}$ and $\mu_{\bf (\hat{x},\hat{u})}$ forms on empirical measure for  $\Theta^N$ whereas $\mu_{\bf \hat{x}}$ forms an empirical measure for $\mu^N$. We further introduce ${\bf \hat{w}} = [\hat{w}^1,\dots,\hat{w}^N]$.  ${\bf \hat{w}}$ is also distributed according to ${\bf P}(\cdot)$. For the joint distribution of ${\bf w,\hat{w}}$, we use a coupling of the form
\begin{align*}
{\bf \Omega}(d {\bf w},d{\bf \hat{w}}) = \Omega^1(dw^1 ,d\hat{w}^1)\times \dots \times \Omega^N(dw^N,d\hat{w}^N).
 \end{align*}
That is, we assume independence over $i\in{1,\dots,N}$, however, an arbitrary coupling is assumed between the distribution of $w^i,\hat{w}^i$.  We will later specify the particular selection of coordinate wise couplings  $\Omega^1,\dots,\Omega^N$. We write
\begin{align*}
& \int \left\|  F(\mu^N,\gamma^N) -\mu_{f(\bf x,u,w)}\right\| {\bf P}(d{\bf w}) \\
&\leq  E\left[  \int \left\|  F(\mu^N,\gamma^N) - \mu_{f ({\bf \hat{x},\hat{u},\hat{w}})}\right\|   +   \left\|   \mu_{f ({\bf \hat{x},\hat{u},\hat{w}})} -   \mu_{f(\bf x,u,w)}\right\| {\bf \Omega}(d{\bf w}, d{\bf \hat{w}})  \right]
\end{align*}
where the expectation is with respect to the random realizations of $(\hat{x}^i, \hat{u}^i)\sim \Theta^N(du,dx)$. The first term corresponds to the expected difference between the empirical measures of $\mu_1=F(\mu^N,\gamma^N)$ and $\mu_1$ itself, and thus is bounded by $M_N$.

For the second term, we note that $\mu_{ f({\bf x},{\bf u},{\bf w})} $ and $ {\mu}_{f(\hat{\bf x}, \hat{\bf u} , \hat{\bf w}) }$ are empirical measures, and thus 
for every given realization of ${\bf w}$ and ${\bf \hat{w}}$, the Wasserstein distance is achieved with a particular permutation of ${ f({\bf x},{\bf u},{\bf w})}$ and ${{f}(\hat{\bf x}, \hat{\bf u} , \hat{\bf w}) }$ combined together. That is, letting $\sigma$ denote a permutation map for the vector ${{f}(\hat{\bf x}, \hat{\bf u} , \hat{\bf w}) }$. we have
\begin{align*}
 \left\|\mu_{ f({\bf x},{\bf u},{\bf w})} - {\mu}_{\hat{f}(\hat{\bf x}, \hat{\bf u} , \hat{\bf w}) }\right\|  = \inf_\sigma \frac{1}{N}\sum_{i=1}^N |f(x^i,u^i,w^i,\mu^N) - \sigma({f}(\hat{x}^i,\hat{u}^i,\hat{w}^i,\mu^N))|.
\end{align*}
 We will however, consider a particular permutation where 
\begin{align*}
&\left\| \int  \mathcal{T}(\cdot|x,u,\mu^N) \mu_{\bf (x,u)}(du,dx) - {\mathcal{T}}(\cdot|x,u,{\mu^N}) \mu_{(\bf \hat{x} , \hat{u})}(du,dx)\right\|\nonumber\\
&= \frac{1}{N}\sum_{i=1}^N \left\| \mathcal{T}(\cdot|x^i,u^i,\mu^N) -\sigma({\mathcal{T}}(\cdot|\hat{x}^i,\hat{u}^u,{\mu}^N))\right\|
\end{align*} 
For the following analysis, we will drop the permutation notation $\sigma$ and assume that the given order of ${f}(\hat{\bf x}, \hat{\bf u} , \hat{\bf w}) $ achieves the Wasserstein distance above. Furthermore, the coupling ${\bf \Omega}$ is assumed to have the same order of coordinate-wise coupling.

 We then write
\begin{align*}
 & \int  \left\|\mu_{ f({\bf x},{\bf u},{\bf w})} - {\mu}_{{f}(\hat{\bf x}, \hat{\bf u} , \hat{\bf w}) }\right\| {\bf \Omega}(d{\bf w}, d\hat{\bf w}) \nonumber\\
&\leq \int \frac{1}{N}\sum_{i=1}^N \left| f(x^i,u^i,w^i, \mu^N) - {f}(\hat{x}^i,\hat{u}^i,\hat{w}^i,{\mu^N}) \right|  {\bf \Omega}(d{\bf w}, d\hat{\bf w})\nonumber \\
& =  \frac{1}{N}\sum_{i=1}^N \int  \left| f(x^i,u^i,w^i, \mu^N) - {f}(\hat{x}^i,\hat{u}^i,\hat{w}^i,{\mu^N}) \right|  {\bf \Omega}(d{\bf w}, d\hat{\bf w})\nonumber \\
&= \frac{1}{N}\sum_{i=1}^N \int  \left| f(x^i,u^i,w^i, \mu^N) - {f}(\hat{x}^i,\hat{u}^i,\hat{w}^i,{\mu}^N) \right|  { \Omega^i}(dw^i, d\hat{w}^i). 
\end{align*}

The analysis thus far, works for any coupling  ${\bf \Omega}(d{\bf w}, d\hat{\bf w})$. In particular, the analysis holds for the coupling that satisfies
\begin{align*}
\| \mathcal{T}(\cdot|x,u,\mu^N) - {\mathcal{T}}(\cdot|\hat{x}^i,\hat{u}^i,{\mu}^N)\| =  \int  \left| f(x^i,u^i,w^i, \mu^N) - {f}(\hat{x}^i,\hat{u}^i,\hat{w}^i,{\mu}^N) \right|  { \Omega^i}(dw^i, d\hat{w}^i).
\end{align*}
for every $i$ for some coordinate-wise coupling $\Omega^i(dw^i,d\hat{w}^i)$. We can then write 
\begin{align*}
 & \int  \left\|\mu_{ f({\bf x},{\bf u},{\bf w})} - {\mu}_{{f}(\hat{\bf x}, \hat{\bf u} , \hat{\bf w}) }\right\| {\bf \Omega}(d{\bf w}, d\hat{\bf w}) \nonumber\\
&\leq \frac{1}{N}\sum_{i=1}^N \int  \left| f(x^i,u^i,w^i, \mu^N) - {f}(\hat{x}^i,\hat{u}^i,\hat{w}^i,{\mu^N}) \right|  { \Omega^i}(dw^i, d\hat{w}^i)\\
& = \int \frac{1}{N}\sum_{i=1}^N \left\|  \mathcal{T}(\cdot|x,u,\mu^N) - {\mathcal{T}}(\cdot|\hat{x}^i,\hat{u}^i,{\mu^N})\right\|\\
&=\left\| \int  \mathcal{T}(\cdot|x,u,\mu^N) \mu_{\bf (x,u)}(du,dx) - \int{\mathcal{T}}(\cdot|x,u,{\mu^N}) \mu_{(\bf \hat{x} , \hat{u})}(du,dx)\right\|.
\end{align*}
We can then write that
\begin{align*}
& \int \left\|  F(\mu^N,\gamma^N) -\mu_{f(\bf x,u,w)}\right\| {\bf P}(d{\bf w}) \\
&\leq  E\left[  \int \left\|  F(\mu^N,\gamma^N) - \mu_{f ({\bf \hat{x},\hat{u},\hat{w}})}\right\|   +   \left\|   \mu_{f ({\bf \hat{x},\hat{u},\hat{w}})} -   \mu_{f(\bf x,u,w)}\right\| {\bf \Omega}(d{\bf w}, d{\bf \hat{w}})  \right]\\
&\leq M_N + E\left[\left\| \int  \mathcal{T}(\cdot|x,u,\mu^N) \mu_{\bf (x,u)}(du,dx) - \int{\mathcal{T}}(\cdot|x,u,{\mu^N}) \mu_{(\bf \hat{x} , \hat{u})}(du,dx)\right\|.\right]\\
& \leq M_N + E\left[\delta_T\left\|\mu_{\bf (\hat{x},\hat{u})} - \mu_{(\bf x,u)}\right\| \right] \leq M_N + \delta_T \bar{M}_N
\end{align*}
where in the last step we used the fact that $\mu_{\bf (\hat{x},\hat{u})}$ is an empirical measure for $\mu_{(\bf x,u)} = \Theta^N$.

We then conclude that for the term (\ref{bound_inf_big}):
\begin{align}\label{bound_inf_big1}
& K_\beta^{*}(\mu^N) - K_\beta^{N,*}(\mu^N) \nonumber\\
& \leq \beta \|K_\beta^*\|_{Lip} \int \left\| F(\mu^N,\gamma^N) - \mu_1^N\right\| \eta(d\mu_1^N|\mu^N,\Theta^N) + \beta \sup_{\mu\in \P_N(\mathds{X})} \left| K_\beta^*(\mu)-  K_\beta^{N,*}(\mu)\right|\nonumber\\
&\leq \beta \|K_\beta^*\|_{Lip} \left( M_N + \delta_T \bar{M}_N\right) +  \beta \sup_{\mu\in \P_N(\mathds{X})} \left| K_\beta^*(\mu)-  K_\beta^{N,*}(\mu)\right|
\end{align}

We now assume that $K_\beta^*(\mu^N) < K_\beta^{N,*}(\mu^N)$. To get an upper bound similar to (\ref{bound_inf_big}), for the finite population problem, we let agents to use the randomized policy $\gamma_\infty$ that is optimal for the infinite population problem, instead of choosing actions that achieves $\Theta^N$ which is the optimal selection for the $N$ population problem for the state distribution $\mu^N$.  Let ${\bf x}$ be such that $\mu_{\bf x}=\mu^N$, we introduce ${{\bf u }}= [{u}^1,\dots,{u}^N]$ where ${u}^i = \gamma_\infty(x^i,v^i)$ for some i.i.d. $v^i$.  Denoting by $\hat{\Theta}^N = \mu_{\bf (x,{u})}$, and  following the steps leading to (\ref{bound_inf_big}), we now write
\begin{align}\label{bound_fin_big}
& K_\beta^{N,*}(\mu^N) - K_\beta^{*}(\mu^N) \leq  \beta \int K_\beta^{N,*}(\mu_1^N)\eta(d\mu_1^N|\mu^N, \hat{\Theta}^N) - \beta K_\beta^*(F(\mu^N,\gamma_\infty))\nonumber\\
& \leq  \beta \int \left| K_\beta^{N,*}\left(\mu_1^N\right) - K_\beta^*(\mu_1^N)\right| \eta(d\mu_1^N|\mu^N,\hat{\Theta}^N)\nonumber\\
&\quad + \beta \int \left|  K_\beta^*(\mu_1^N) - K_\beta^{*}\left( F(\mu^N,\gamma_\infty)  \right)\right| \eta(d\mu_1^N|\mu^N,\hat{\Theta}^N)\nonumber\\
& \leq  \beta \sup_{\mu\in \P_N(\mathds{X})} \left| K_\beta^*(\mu)-  K_\beta^{N,*}(\mu)\right|+ \beta \|K_\beta^*\|_{Lip} \int \left\| F(\mu^N,\gamma_\infty) - \mu_1^N\right\| \eta(d\mu_1^N|\mu^N,\hat{\Theta}^N)
\end{align}
Following almost identical steps as the first case, one can show that
\begin{align*}
& \int \left\| F(\mu^N,\gamma_\infty) - \mu_1^N\right\| \eta(d\mu_1^N|\mu^N,\hat{\Theta}^N)\\
& \leq M_N + \delta_T  E\left[\left\|\mu_{\bf (\hat{x},\hat{u})} - \mu_{(\bf x,u)}\right\| \right] 
\end{align*} 
where $\hat{x}^i \sim \mu^N$, $\mu_{\bf x}=\mu^N$ and $u^i=\gamma_\infty(x^i,v^i)$, $\hat{u}^i = \gamma_\infty(\hat{x}^i,v^i)$, and the expectation above is with respect to the random selections of $\hat{x}^i$ and $v^i$. Note that $u^i$ and $\hat{u}^i$ uses the same randomization $v^i$, hence averaging over the distribution of $v^i$, we can write that
\begin{align*}
 E\left[\left\|\mu_{\bf (\hat{x},\hat{u})} - \mu_{(\bf x,u)}\right\| \right]  \leq& E \left[ \left\| \gamma_\infty(du|x)\mu_{\bf x}(dx) - \gamma_\infty(du|x)\mu_{\bf \hat{x}}(dx)\right\|\right]\\
&\leq E \left[  \|\mu_{\bf x} - \mu_{\bf \hat{x}}\| \right] \leq M_N.
\end{align*}
In particular, we can conclude that the bound (\ref{bound_fin_big}) can be concluded as:
\begin{align}\label{bound_fin_big1}
& K_\beta^{*}(\mu^N) - K_\beta^{N,*}(\mu^N)  \leq \beta \|K_\beta^*\|_{Lip}\left( M_N +  \delta_T M_N \right) +  \beta \sup_{\mu\in \P_N(\mathds{X})} \left| K_\beta^*(\mu)-  K_\beta^{N,*}(\mu)\right|.
\end{align}
Thus, noting that $M_N\leq \bar{M}_N$, and combining (\ref{bound_inf_big1}) and (\ref{bound_fin_big1}), we can write
\begin{align*}
& |K_\beta^{*}(\mu^N) - K_\beta^{N,*}(\mu^N) | \leq \beta \|K_\beta^*\|_{Lip}\left( \bar{M}_N +  \delta_T \bar{M}_N \right) +  \beta \sup_{\mu\in \P_N(\mathds{X})} \left| K_\beta^*(\mu)-  K_\beta^{N,*}(\mu)\right|.
\end{align*}
Rearranging the terms and taking the supremum on the left hand side over $\mu^N\in \P_N(\mathds{X})$, we can write
\begin{align*}
&\sup_{\mu\in \P_N(\mathds{X})} |K_\beta^{*}(\mu) - K_\beta^{N,*}(\mu) | \leq \frac{\beta \|K_\beta^*\|_{Lip} (1-\delta_T) \bar{M}_N}{1-\beta}
\end{align*}
which proves the result together with $\|K_\beta^*\|_{Lip}\leq \frac{C}{1-\beta K}$ and $\delta_T\leq1$.

\bibliographystyle{plain}

\bibliography{mfc_bibliography}

\end{document}